\documentclass[11pt,a4paper]{article}
\usepackage[T1]{fontenc}
\usepackage[utf8]{inputenc}
\usepackage[margin=1in]{geometry}
\usepackage{amsmath,amssymb,amsthm,mathtools,bm}
\usepackage{booktabs,array,enumitem,microtype}
\usepackage{graphicx,subcaption,algorithm,algpseudocode}
\usepackage{float,placeins,xcolor,tikz}
\usepackage[round,authoryear]{natbib}
\usepackage[hidelinks]{hyperref}
\usetikzlibrary{arrows.meta,positioning}
\graphicspath{{figures/}{./}}
\allowdisplaybreaks
\numberwithin{equation}{section}
\theoremstyle{plain}
\newtheorem{theorem}{Theorem}[section]
\newtheorem{proposition}[theorem]{Proposition}
\newtheorem{corollary}[theorem]{Corollary}
\newtheorem{lemma}[theorem]{Lemma}
\theoremstyle{definition}
\newtheorem{assumption}[theorem]{Assumption}
\newtheorem{definition}[theorem]{Definition}

\newcommand{\X}{\mathbb X}

\newcommand{\E}{\mathbb E}
\newcommand{\Pp}{\mathbb P}

\newcommand{\PPP}{\operatorname{PPP}}
\newcommand{\Pois}{\operatorname{Poisson}}
\newcommand{\Bern}{\operatorname{Bernoulli}}
\newcommand{\Bin}{\operatorname{Binomial}}
\newcommand{\Unif}{\operatorname{Unif}}
\newcommand{\dd}{\,\mathrm d}

\newcommand{\ind}{\mathbf 1}

\newcommand{\Torus}{\mathbb T^2}
\title{Exact Likelihood-Coin Poisson Sampling for Bayesian Inverse Problems with Sharp Complexity Bounds}
\author{Zhiliang Deng\thanks{School of Mathematical Sciences, University of Electronic Science and Technology of China. Email: dengzhl@uestc.edu.cn}\and Xiaomei Yang\thanks{School of Mathematics, Southwest Jiaotong University. Email: yangxiaomath@swjtu.edu.cn}}
\date{}

\begin{document}

\maketitle

\begin{abstract}
We develop an exact posterior-sampling framework for Bayesian inverse problems when selected bounded forward observables can be accessed through Bernoulli events. A Bernstein--Poisson construction converts these forward coins into scaled Gaussian likelihood coins, and thinning an inflated prior Poisson point process yields posterior atoms that are iid conditional on their number. For independent Gaussian observations, we derive an exact mean-work identity for the implemented early-stopped factory and sharp small-noise complexity laws governed by local prior-predictive mass near the exact-fit set; a factorial-moment construction extends the likelihood factory to correlated Gaussian errors. The posterior algorithm is model-agnostic once Bernoulli access is available. As one continuum realization, we use Feynman--Kac sampling, Poisson killing, and lazy random-series evaluation for a bounded elliptic resolvent problem with a function-valued coefficient. Numerical experiments validate the forward and likelihood coins, support the predicted work regimes, and demonstrate posterior sampling without deterministic spatial discretization or fixed parameter truncation in the target. A matched-accuracy benchmark against finite-difference prior rejection illustrates how deterministic discretization bias changes the posterior accuracy--cost balance.
\end{abstract}

\par

\section{Introduction}

Bayesian inverse problems combine uncertain parameters, noisy data, and a forward model through a posterior probability measure. In PDE-constrained problems, posterior computation is often dominated by repeated evaluations of the parameter-to-observation map \cite{KaipioSomersalo2005,Stuart2010}. Markov chain Monte Carlo therefore inherits the cost of a forward solve at each proposal. Exact-likelihood methods based on stochastic estimators can avoid deterministic likelihood evaluation, but typically retain a Markov-chain architecture; pseudo-marginal methods are a canonical example \cite{AndrieuRoberts2009}. Bernoulli factories provide a complementary exact-simulation mechanism when events with unknown probabilities can be generated directly \cite{Huber2016,KeaneOBrien1994,LatuszynskiEtAl2011,NacuPeres2005}. They have been used for perfect or exact sampling, likelihood-free acceptance decisions, and scalable Bayesian constructions \cite{FlegalHerbei2012,StumpfFetizonGoncalves2025,StumpfFetizonEtAl2026,VatsEtAl2022}.

The present paper uses this primitive in a different way. We assume that selected bounded forward observables can be queried through exact Bernoulli events. Rather than average those events to approximate a forward solution and then evaluate a likelihood, we use them directly to manufacture a likelihood coin. That coin thins an inflated Poisson point process whose base measure is the prior. Classical independent thinning then produces a point process with the unnormalized posterior as its intensity; conditional on its cardinality, the retained locations are independent posterior draws. The posterior-as-intensity identity itself is standard point-process theory \cite{Kingman1993,LastPenrose2017} and is not claimed as new. An earlier preprint by the present authors used this viewpoint together with kernel and Gaussian-mixture approximations \cite{DengEtAl2025PPP}; here the likelihood decision is generated directly from stochastic forward events, without a numerical likelihood evaluation or a Markov transition.

At the core of the method is an explicit Gaussian likelihood factory. For conditionally independent observations, a Bernstein statistic gives an unbiased representation of a squared residual, a deterministic compensation makes the statistic Bernoulli-admissible, and a Poisson zero-event transform produces the Gaussian factor exactly up to a known scale. Independent scalar factors then assemble the vector likelihood. To show that factorization is not essential to the Bernoulli-access principle, we also give a compact multivariate factorial-moment extension for correlated Gaussian errors. The sharp work analysis is deliberately developed for the independent architecture, where the stopping mechanism can be characterized cleanly.

The exact construction also permits a sharp analysis of its actual computational work. A conservative full-batch calculation captures the basic order scale but ignores two exact savings: a Bernstein event may be resolved before all forward coins are revealed, and a candidate may be rejected before all auxiliary events are processed. We derive an exact mean-work identity for the implemented stopped sampler. Under a local small-ball law for the prior predictive distribution, the resulting small-noise rate is controlled by the prior mass near the exact-fit set. The transition is governed by an effective local predictive dimension rather than by the nominal dimension of the parameter space. This connects the implementation-level stopping rules directly to a predictive-geometric complexity law.

The posterior algorithm is not tied to a PDE class: its model-facing input is only the Bernoulli-access interface. To demonstrate that this interface is non-vacuous for an infinite-dimensional inverse problem, we construct it explicitly for a bounded class of elliptic resolvent problems. A Feynman--Kac representation turns the reaction factor into Poisson killing and the source factor into a terminal Bernoulli mark. Function-valued unknowns represented by absolutely convergent random series can be queried lazily: coefficients are revealed only until a deterministic tail enclosure resolves the Bernoulli decision. Thus the mathematical posterior target need not contain either a deterministic spatial mesh or a fixed parameter truncation.

The numerical study follows this structure. A two-parameter elliptic benchmark checks the continuum forward coin and then runs actual Bernstein--Poisson likelihood coins under both independent and correlated Gaussian observation errors; calibration compares compensated coin frequencies with exactly known Gaussian likelihood values. A predictive-dimension experiment tests the sharp stopped-work rates using the independent-observation implementation analyzed in the complexity theory. Finally, a heterogeneous reaction-coefficient problem combines the likelihood factory, Poisson killing, and lazy random-series evaluation. In that last experiment the validation target is the deterministic reference posterior, not pointwise recovery of the synthetic truth from a deliberately small data set. A final matched-accuracy benchmark uses the same heterogeneous reaction-coefficient model, observation locations, and data to compare continuum likelihood-coin sampling with finite-difference prior rejection. It includes deterministic posterior discretization bias in the RMS-error budget, accounts for the prior-rejection penalty on both sides, and selects the least-cost admissible FD grid at each target accuracy. The resulting crossover is reported only as a problem- and implementation-specific benchmark, not as a universal efficiency threshold.

\paragraph{Contributions and computational scope.}
The contribution is an exact-simulation methodology rather than a PDE-specific solver. Its model-facing requirement is Bernoulli access to bounded forward observables. Section~\ref{sec:pde} gives one nontrivial continuum realization of this interface, whereas the likelihood factory and posterior Poisson thinning are independent of that particular PDE representation. The numerical study therefore separates calibration of the probabilistic components, validation of the sharp work theory, a function-valued continuum realization, and a matched-accuracy computational benchmark.

The paper is organized as follows. Section~\ref{sec:poisson} develops posterior Poissonization. Section~\ref{sec:factory} constructs the Gaussian likelihood coins, records the correlated extension, and proves the sharp actual-work bounds for the independent sampler. Section~\ref{sec:pde} gives one concrete PDE realization of the Bernoulli interface, including function-valued parameters. Section~\ref{sec:numerics} presents the numerical validation and PDE cost comparison, and Section~\ref{sec:conclusion} summarizes the computational contribution and its implications.

\section{Posterior Poissonization from likelihood coins}
\label{sec:poisson}

This section isolates the point-process part of the method. On a general measurable parameter space, Bayesian updating weights the prior by the likelihood, while independent thinning weights a Poisson intensity by a retention probability. The construction below identifies these two weights; the later sections show how to generate the required retention event without evaluating the likelihood.

\subsection{Poisson point processes on a general parameter space}

Let $(\X,\mathcal X)$ be a measurable space and let $\Lambda$ be a finite measure on it. A PPP with intensity measure $\Lambda$ is a random counting measure $\Xi=\sum_{i=1}^{N}\delta_{Q_i}$,
with the following property: for pairwise disjoint measurable sets $A_1,\ldots,A_m\in\mathcal X$, the counts $\Xi(A_1),\ldots,\Xi(A_m)$ are independent and
$\Xi(A_j)\sim\Pois\bigl(\Lambda(A_j)\bigr)$ for $j=1,\ldots,m$.
We write $\Xi\sim\PPP(\Lambda)$. Standard references include \cite{Kingman1993,LastPenrose2017}. Since $\Lambda(\X)<\infty$, the process may equivalently be generated by
\begin{equation}
 N\sim\Pois\bigl(\Lambda(\X)\bigr),
 \qquad
 Q_i\mid N\stackrel{\rm iid}{\sim}\frac{\Lambda}{\Lambda(\X)}.
 \label{eq:finite-ppp-iid}
\end{equation}
Thus a finite PPP contains two pieces of information: the total mass of its intensity is encoded by the random cardinality, whereas the normalized intensity is encoded by the conditional location law.

Independent thinning is a fundamental property of PPPs. Let $r:\X\to[0,1]$ be measurable and, conditional on $\Xi$, retain each atom $Q_i$ independently with probability $r(Q_i)$. The retained process is a PPP with intensity measure $r\Lambda$, defined by $(r\Lambda)(A)=\int_A r(z)\,\mathrm d\Lambda(z)$, $A\in\mathcal X$.
Thus thinning multiplies the original intensity by the retention function.

Two further standard properties are useful computationally. If $\Xi_1\sim\PPP(\Lambda_1)$ and $\Xi_2\sim\PPP(\Lambda_2)$ are independent, then their superposition $\Xi_1+\Xi_2$ is $\PPP(\Lambda_1+\Lambda_2)$. Conversely, restrictions of a PPP to disjoint measurable subsets are independent. These identities justify splitting the master population into independently processed batches without changing its law and make atom-wise processing embarrassingly parallel. More generally, independent marking of the atoms produces a marked PPP, and thinning is the special case in which the mark records only ``keep'' or ``discard.'' Thus the posterior algorithm can be viewed as a marked point-process construction whose retention mark is manufactured from stochastic PDE events.

\subsection{Bayesian updating as selective survival}

Let $q\in\X$
denote the unknown, and let $\mu_0$ be its prior probability measure.
For fixed data $y$, let $L(q; y)\in[0,1]$ be a measurable reduced
likelihood and assume $0<Z(y):=\int_\X L(q; y)\,\mathrm d\mu_0(q)<\infty$.
Define
\[
\mathrm d\nu^y(q)=L(q;y)\,\mathrm d\mu_0(q),
 \qquad
 \frac{\mathrm d\mu^y}{\mathrm d\mu_0}(q)
 =\frac{L(q;y)}{Z(y)}.
\]
If a PPP with intensity $\mu_0$ could be thinned with retention probability $L(q;y)$, its retained intensity would be $\nu^y$. In a PDE inverse problem the likelihood value may be unavailable without a forward solve, so we replace the numerical value by a Bernoulli decision having the required success probability.

\begin{definition}[Scaled likelihood coin]
A scaled likelihood-coin oracle with known factor $c\in(0,1]$ returns, when queried at a parameter value $q\in\X$, a Bernoulli variable $C_q\in\{0,1\}$ satisfying
\begin{equation}
 \Pp(C_q=1)=cL(q;y).
 \label{eq:scaled-coin}
\end{equation}
Repeated queries use independent auxiliary randomness. The numerical value of $L(q;y)$ need not be available.
\end{definition}

We use a positive parameter $\gamma$ as a Poissonization scale.
It controls the overall size of the generated point process; in particular,
the retained process will have expected cardinality $\gamma Z(y)$.

\begin{theorem}[Likelihood-coin posterior Poissonization]
\label{thm:coin-ppp}
Suppose that the scaled likelihood-coin oracle
\eqref{eq:scaled-coin} is available. For any Poissonization scale
$\gamma>0$, generate
$\Xi_\gamma=\sum_{i=1}^{M_\gamma}\delta_{Q_i}\sim\PPP\!\left(\frac{\gamma}{c}\mu_0\right)$.
Conditional on $\Xi_\gamma$, query the oracle independently at each atom
$Q_i$ and denote the resulting Bernoulli mark by $C_i$, so that
$\Pp(C_i=1\mid Q_i)=cL(Q_i; y)$.
Define the retained point process by
$\eta_\gamma:=\sum_{i=1}^{M_\gamma}C_i\,\delta_{Q_i}$.
Then
$\eta_\gamma\sim\PPP(\gamma\nu^y)$.
Consequently,
\begin{equation}
 N_\gamma:=\eta_\gamma(\X)\sim\Pois\bigl(\gamma Z(y)\bigr),
 \label{eq:count-poisson}
\end{equation}
and, conditional on $N_\gamma=k\ge1$,
\[
\eta_\gamma\mid\{N_\gamma=k\}\overset{d}{=}\sum_{i=1}^k\delta_{Q_i^\star},
 \qquad
 Q_i^\star\stackrel{\rm iid}{\sim}\mu^y.
\]
Moreover, $\widehat Z_\gamma=\frac{N_\gamma}{\gamma}$
is unbiased for $Z(y)$ with $\operatorname{Var}(\widehat Z_\gamma)=\frac{Z(y)}{\gamma}$.
\end{theorem}

\begin{proof}
Conditional on the master process, the marks $C_i$ are independent and
the retention probability of an atom at $z\in\X$ is $cL(z;y)$.
Therefore, by the independent-thinning theorem, the point process $\eta_\gamma=\sum_{i=1}^{M_\gamma}C_i\delta_{Q_i}$
has intensity $\frac{\gamma}{c}\,cL(z; y)\,\mu_0(\mathrm dz)=\gamma\,\nu^y(\mathrm dz)$.
Hence $\eta_\gamma\sim\PPP(\gamma\nu^y)$.
Its total intensity is $\gamma Z(y)$, which gives
\eqref{eq:count-poisson}. Conditional on $N_\gamma=k$, the finite-PPP
representation \eqref{eq:finite-ppp-iid} gives iid locations with law $\frac{\nu^y}{Z(y)}=\mu^y$.
Finally, since $N_\gamma\sim\Pois(\gamma Z(y))$,
$\E\widehat Z_\gamma=Z(y)$, $\operatorname{Var}(\widehat Z_\gamma)=\frac{Z(y)}{\gamma}$.
\end{proof}

Conditional on $N_\gamma=k$, Theorem~\ref{thm:coin-ppp} makes the retained
locations an ordinary iid posterior sample.  Hence, for any square-integrable posterior functional $\varphi$, the empirical posterior average has conditional variance
\begin{align*}
\operatorname{Var}\left(\frac{1}{k}\sum_{i=1}^k \varphi(Q_i^*)\mid N_\gamma=k\right)=\frac{\operatorname{Var}_{\mu^y}(\varphi)}{k}.
\end{align*}
The same theorem
also makes the rejection cost explicit: $\E M_\gamma=\gamma/c$ and
$\E N_\gamma=\gamma Z(y)$, so the ratio of their means is $[cZ(y)]^{-1}$.
This ratio records the average candidate count associated with the outer thinning step; the forward-oracle work per processed candidate is analyzed separately in Section~\ref{sec:factory}.

The complete construction is summarized in
Figure~\ref{fig:pipeline}. A master PPP drawn from the prior supplies the
candidate atoms. Stochastic PDE representations generate the required
forward coins, the Bernstein--Poisson factory converts them into a scaled
likelihood coin, and independent thinning produces the posterior PPP.
Sections~\ref{sec:factory} and~\ref{sec:pde} develop the two middle steps.

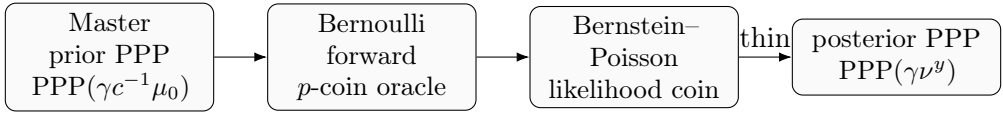
\begin{figure}[t]
\centering
\begin{tikzpicture}[node distance=8mm and 7mm,>=Latex,
box/.style={draw,rounded corners,align=center,minimum height=10mm,text width=25mm,fill=gray!4,font=\small}]
\node[box] (prior) {Master prior PPP\\$\PPP(\gamma c^{-1}\mu_0)$};
\node[box,right=of prior] (pde) {Bernoulli forward\\$p$-coin oracle};
\node[box,right=of pde] (factory) {Bernstein--Poisson\\likelihood coin};
\node[box,right=of factory] (post) {posterior PPP\\$\PPP(\gamma\nu^y)$};
\draw[->] (prior)--(pde);
\draw[->] (pde)--(factory);
\draw[->] (factory)--node[above]{thin}(post);
\end{tikzpicture}
\caption{Likelihood-coin Poisson sampling.
A Bernoulli forward oracle supplies the input coins for a scaled likelihood coin, which thins the master prior PPP into the target posterior PPP.}
\label{fig:pipeline}
\end{figure}

\section{Bernstein--Poisson likelihood coins for Gaussian observations}
\label{sec:factory}

The posterior construction of Section~\ref{sec:poisson} requires a scaled likelihood coin, while the stochastic PDE mechanisms developed later provide only Bernoulli access to individual bounded forward observables. We first treat conditionally independent Gaussian observation errors, for which the likelihood factory, the exact stopping rules, and the sharp work analysis share a common componentwise architecture. A compact correlated-Gaussian extension is given in Section~\ref{subsec:correlated-factory}. Conditional on the unknown $q$, let
\begin{equation}
 y_j=p_j(q)+\varepsilon_j,
 \qquad
 \varepsilon_j\sim\mathcal N(0,\sigma_j^2),
 \qquad j=1,\ldots,J,
 \label{eq:obsmodel}
\end{equation}
where $p_j(q)$, $y_j\in[0,1]$ and the noise components are independent. The quantities $p_j(q)$ need not be independent or otherwise unrelated: conditional on $q$ they are deterministic values, and in the PDE applications they are typically different observations of the same solution. After removing the Gaussian normalizing constants, the reduced likelihood is
\begin{equation}
 L(q;y)
 =\exp\!\left\{-\sum_{j=1}^J a_j[p_j(q)-y_j]^2\right\},
 \qquad a_j=\frac{1}{2\sigma_j^2}.
 \label{eq:multi-L}
\end{equation}
The values $p_j(q)$ are not supplied to the sampler. For a queried parameter value, the sampler can only call an oracle that returns independent Bernoulli variables with the corresponding success probability. Since the vector likelihood factorizes, it is enough to construct a coin for one scalar factor
\begin{equation}
 \ell_y(p)=e^{-a(p-y)^2},
 \qquad p, \, y\in[0,1],\quad a>0,
 \label{eq:one-likelihood}
\end{equation}
from independent $p$-coins, without evaluating $p$.

The centered scalar factor $y=0$ is unusually simple: two independent $p$-coins produce a $p^2$-event, and a Poisson zero-event transform then gives $e^{-ap^2}$. For a general component $y$, however, the term $(p-y)^2=p^2-2yp+y^2$ contains a signed cross term and cannot itself be used as an event probability. The construction below resolves precisely this scalar obstruction; the full vector likelihood is recovered by multiplying the resulting independent component coins.

For an integer $n\ge2$, let $X_1, \ldots, X_n\stackrel{\rm iid}{\sim}\Bern(p)$, $K=\sum_{i=1}^nX_i\sim\Bin(n, p)$,
and define
\begin{equation}
 \beta_{k,n}(y)
 =y^2-\frac{2yk}{n}+\frac{k(k-1)}{n(n-1)},
 \qquad k=0,\ldots,n.
 \label{eq:beta}
\end{equation}

\begin{lemma}
\label{lem:square}
For $K\sim\Bin(n, p)$, $\E\,\beta_{K, n}(y)=(p-y)^2$.
\end{lemma}
\begin{proof}
Use $\E K=np$ and $\E[K(K-1)]=n(n-1)p^2$ in \eqref{eq:beta}.
\end{proof}

The statistic $\beta_{K,n}(y)$ is unbiased for the squared residual but may take negative values. Moreover, setting $p=y$ in Lemma~\ref{lem:square} gives
$\E\,\beta_{K, n}(y)=0$, $K\sim\Bin(n, y)$,
and hence $\min_{0\le k\le n}\beta_{k,n}(y)\le0$.
We therefore define the smallest nonnegative upward shift
\[
b_n(y):=-\min_{0\le k\le n}\beta_{k,n}(y)\ge0.
\]
Set
\begin{equation}
c_{k,n}(y)=\beta_{k,n}(y)+b_n(y),
\qquad
C_n(y)=\max\left\{1,\max_{0\le k\le n}c_{k,n}(y)\right\}.
\label{eq:ck}
\end{equation}
The shift makes $c_{k,n}(y)$ nonnegative, while the normalization by
$C_n(y)$ places it in $[0,1]$. We can therefore convert the random statistic
$c_{K,n}(y)$ into a Bernoulli event. After generating $K$, define the
auxiliary Bernoulli variable
\begin{equation}
H_{n,y}\mid K=k
\sim\Bern\left(\frac{c_{k,n}(y)}{C_n(y)}\right).
\label{eq:Hcoin}
\end{equation}
By conditioning on $K$ and using Lemma~\ref{lem:square},
\[
\Pp(H_{n, y}=1)
=\frac{b_n(y)+(p-y)^2}{C_n(y)}.
\]
For \(y\in[0,1]\), one in fact has
\(0\le c_{k,n}(y)\le1\), and hence \(C_n(y)=1\);
see Appendix~\ref{app:bernstein-compensation}.
We retain $C_n(y)$ in the formulas because it makes the admissibility of the Bernoulli construction explicit and allows the same argument to remain valid under simple rescalings.

Crucially, the Bernstein coefficients \(c_{k,n}(y)\) are completely
determined by the observed datum \(y\), the order \(n\), and the realized
success count \(k\); they do not require the numerical value of \(p\).
The forward model enters the construction only through
\(K\sim\Bin(n,p)\), which is generated by repeated calls to the
\(p\)-coin oracle.

\begin{theorem}[Bernstein--Poisson Gaussian likelihood coin]
\label{thm:onefactory}
Let $R\sim\Pois(aC_n(y))$. Independently generate $R$ copies of $H_{n,y}$ from \eqref{eq:Hcoin}, and define the Gaussian likelihood coin $G_{n, y}$ by setting $G_{n, y}=1$ if and only if all $R$ copies are zero. Then
\begin{equation}
 \Pp(G_{n, y}=1)
 =e^{-ab_n(y)}e^{-a(p-y)^2}.
 \label{eq:one-factory-prob}
\end{equation}
Thus $G_{n, y}$ is a scaled likelihood coin for \eqref{eq:one-likelihood}, with known scaling factor $e^{-ab_n(y)}$.
\end{theorem}

\begin{proof}
Conditioning on $R=r$ gives
\[
 \Pp(G_{n, y}=1\mid R=r)
 =\left\{1-\frac{b_n(y)+(p-y)^2}{C_n(y)}\right\}^{r}.
\]
Using the probability generating function of a Poisson variable,
$\E[s^R]=\exp\{aC_n(y)(s-1)\}$, with
$s=1-[b_n(y)+(p-y)^2]/C_n(y)$ proves \eqref{eq:one-factory-prob}.
\end{proof}

The deterministic compensation can be evaluated explicitly.

\begin{proposition}
\label{prop:bn}
Let $k_*(y, n)=y(n-1)+\frac12$, $d_n(y)=\operatorname{dist}(k_*(y, n), \mathbb Z)$.
For $y\in[0,1]$ and $n\ge2$,
\begin{equation}
 b_n(y)=\frac{y(1-y)}{n}+\frac{\frac14-d_n(y)^2}{n(n-1)}.
 \label{eq:bnformula}
\end{equation}
Consequently,
\begin{equation}
 \frac{y(1-y)}{n}\le b_n(y) \le \frac{y(1-y)}{n}+\frac{1}{4n(n-1)},
 \label{eq:bnbounds}
\end{equation}
and
\[
b_n(y)=\frac{y(1-y)}{n}+O(n^{-2}).
\]
\end{proposition}

\begin{proof}
As a function of a real variable $k$, \eqref{eq:beta} is a convex quadratic with coefficient $[n(n-1)]^{-1}$ and minimizer $k_*$. Completing the square gives
\[
 \beta_{k,n}(y)
 =-\frac{y(1-y)}{n}-\frac{1}{4n(n-1)}
 +\frac{(k-k_*)^2}{n(n-1)}.
\]
The nearest integer to $k_*\in[1/2,n-1/2]$ belongs to $\{0,\ldots,n\}$. Taking the discrete minimum and changing sign gives \eqref{eq:bnformula}; \eqref{eq:bnbounds} follows from $0\le d_n\le1/2$.
\end{proof}

\subsection{Assembly of the vector likelihood}

The preceding subsection constructs an exact scaled likelihood coin for a single scalar Gaussian observation. We now assemble these scalar constructions for a vector of conditionally independent observations. Choose an order $n_j\ge2$ for each observation component and construct independent coins $G_{n_j,y_j}$ by Theorem~\ref{thm:onefactory}. Define
$B_{\bm n}(y)=\sum_{j=1}^J a_jb_{n_j}(y_j)$.
Here $a_j=(2\sigma_j^2)^{-1}$ is the componentwise counterpart of the scalar coefficient $a$ used above.
Since the Gaussian likelihood factorizes across the observation components, the logical AND of the $J$ component coins has success probability
\begin{equation}
e^{-B_{\bm n}(y)}L(q;y).
\label{eq:multi-scaled}
\end{equation}

\begin{corollary}
\label{cor:gaussian-ppp}
Generate
$\Xi_\gamma\sim\PPP\left(\gamma e^{B_{\bm n}(y)}\mu_0\right)$,
write the master process as $\Xi_\gamma=\sum_i\delta_{Q_i}$ and retain an atom $Q_i$ if all component likelihood coins generated from the $p_j(Q_i)$-oracles succeed. Then
$\eta_\gamma\sim\PPP(\gamma\nu^y)$.
No numerical value of $p_j(q)$, $L(q; y)$, or $Z(y)$ is required by the sampler.
\end{corollary}

\begin{proof}
Equation \eqref{eq:multi-scaled} gives a scaled likelihood coin with factor $c=e^{-B_{\bm n}(y)}$. The result follows from Theorem~\ref{thm:coin-ppp}.
\end{proof}

\subsection{A compact extension to correlated Gaussian observations}
\label{subsec:correlated-factory}

The componentwise construction above exploits likelihood factorization. This is the architecture used in the sharp work analysis below, but factorization is not essential to the Bernoulli-access principle. To make this point without introducing a second complexity theory, we record a multivariate construction for
\[
 y=p(q)+\varepsilon,\qquad
 \varepsilon\sim\mathcal N(0,\sigma^2\Sigma),\qquad
 \Omega=\Sigma^{-1}\succ0,
\]
with $p(q)=(p_1(q),\ldots,p_J(q))^\top\in[0,1]^J$ and $a_\sigma=(2\sigma^2)^{-1}$. The reduced likelihood is
\[
 L_\Sigma(q;y)
 =\exp\{-a_\sigma[p(q)-y]^\top\Omega[p(q)-y]\}.
\]
Choose orders $n_j\ge2$ and, using independent auxiliary oracle randomness, draw
\[
 K_j\sim\Bin(n_j,p_j(q)),\qquad
 \widehat p_j=\frac{K_j}{n_j},\qquad
 \widehat p_j^{[2]}=\frac{K_j(K_j-1)}{n_j(n_j-1)}.
\]
Define
\[
\begin{aligned}
 \beta_{\bm n}^{\Omega}(K;y)
 ={}&y^\top\Omega y-2y^\top\Omega\widehat p
 +\sum_{j=1}^J\Omega_{jj}\widehat p_j^{[2]}
 +2\sum_{1\le i<j\le J}\Omega_{ij}\widehat p_i\widehat p_j .
\end{aligned}
\]
The cross-products are unbiased because the auxiliary counts are independent conditional on $q$.

\begin{theorem}[Correlated-Gaussian likelihood coin]
\label{thm:correlated-factory}
For every $q$,
\[
\E\,\beta_{\bm n}^{\Omega}(K;y)=[p(q)-y]^\top\Omega[p(q)-y].
\]
Let
\[
 b_{\bm n}^{\Omega}(y)
 =\left[-\min_{0\le k_j\le n_j}\beta_{\bm n}^{\Omega}(k;y)\right]_+,
 \qquad
 C_{\bm n}^{\Omega}(y)
 =\max\left\{1,\max_k\bigl[\beta_{\bm n}^{\Omega}(k;y)+b_{\bm n}^{\Omega}(y)\bigr]\right\}.
\]
Conditional on $K$, generate
\[
 H_{\bm n,y}^{\Omega}\sim
 \Bern\!\left(
 \frac{\beta_{\bm n}^{\Omega}(K;y)+b_{\bm n}^{\Omega}(y)}
 {C_{\bm n}^{\Omega}(y)}\right).
\]
If $R\sim\Pois(a_\sigma C_{\bm n}^{\Omega}(y))$ and
$G_{\bm n,y}^{\Omega}=1$ exactly when all $R$ independent copies of
$H_{\bm n,y}^{\Omega}$ are zero, then
\[
 \Pp(G_{\bm n,y}^{\Omega}=1)
 =e^{-a_\sigma b_{\bm n}^{\Omega}(y)}L_\Sigma(q;y).
\]
Hence correlated Gaussian observation errors can be handled without factorizing the likelihood or evaluating the forward probabilities numerically.
\end{theorem}

\begin{proof}
The factorial moments give $\E\widehat p_j=p_j(q)$ and
$\E\widehat p_j^{[2]}=p_j(q)^2$, while independence of the auxiliary counts gives
$\E(\widehat p_i\widehat p_j)=p_i(q)p_j(q)$ for $i\ne j$. Thus the displayed quadratic statistic is unbiased. The shift and normalization make its affine rescaling a valid Bernoulli probability, so
\[
 \Pp(H_{\bm n,y}^{\Omega}=1)
 =\frac{b_{\bm n}^{\Omega}(y)+[p(q)-y]^\top\Omega[p(q)-y]}
 {C_{\bm n}^{\Omega}(y)}.
\]
The Poisson probability-generating function then gives the stated success probability.
\end{proof}

For large $J$, the discrete minimum need not be enumerated. With
$\delta_j=\Omega_{jj}/(n_j-1)$, define
\[
\widetilde\beta_{\bm n}^{\Omega}(z;y)
=(z-y)^\top\Omega(z-y)-\sum_{j=1}^J\delta_j z_j(1-z_j),
\qquad z\in[0,1]^J,
\]
and let
$\bar b_{\bm n}^{\Omega}(y)=[-\min_{z\in[0,1]^J}
\widetilde\beta_{\bm n}^{\Omega}(z;y)]_+$.
Because the count grid is contained in $[0,1]^J$,
$\bar b_{\bm n}^{\Omega}(y)\ge b_{\bm n}^{\Omega}(y)$, so this continuous
shift is safe. The minimization is strictly convex since its Hessian is
$2\Omega+2\operatorname{diag}(\delta_1,\ldots,\delta_J)\succ0$; moreover, for
fixed $y\in(0,1)^J$ and $n_{\min}=\min_j n_j\to\infty$,
\[
b_{\bm n}^{\Omega}(y),\ \bar b_{\bm n}^{\Omega}(y)
=\sum_{j=1}^J\frac{\Omega_{jj}y_j(1-y_j)}{n_j}
+O(n_{\min}^{-2}).
\]
The proof is given in Appendix~\ref{app:correlated-shift}.

This extension shows that the likelihood-coin mechanism is not tied to diagonal observation covariance. We do not develop a parallel sharp-work theory for the multivariate factory. The remainder of this section returns to the independent factorized architecture, for which scalar range stopping and componentwise Poisson streams expose the actual work exactly.

\subsection{Early stopping and order selection}

A direct implementation of $H_{n,y}$ draws all $n$ forward coins, but many outcomes can be decided earlier. Write $\widetilde c_{k,n}(y)=\frac{c_{k,n}(y)}{C_n(y)}\in[0,1]$,
draw $V\sim\Unif(0,1)$, and note that
$H_{n,y}=\ind_{\{V<\widetilde c_{K,n}(y)\}}$. After $s$ forward flips with $k$ successes, the final total $K$ can only lie in $\{k,\ldots,k+n-s\}$. With $\ell=n-s$, define
\[
\widetilde c_{\min}(k,\ell)
 =\min_{j=k,\ldots,k+\ell}\widetilde c_{j,n}(y),
 \qquad
 \widetilde c_{\max}(k,\ell)
 =\max_{j=k,\ldots,k+\ell}\widetilde c_{j,n}(y).
\]
If $V<\widetilde c_{\min}(k,\ell)$, the final event is already true for every possible continuation; if $V\ge\widetilde c_{\max}(k,\ell)$, it is already false.

\begin{lemma}
\label{lem:early}
The range test above has exactly the same output as first drawing all $n$ forward coins, computing $K$, and returning $\ind_{\{V<\widetilde c_{K,n}(y)\}}$. It therefore leaves the law of $H_{n,y}$ unchanged.
\end{lemma}

\begin{proof}
Early stopping occurs only when the indicator has the same value for every final success count that remains reachable. No admissible continuation can therefore change the result.
\end{proof}

The Bernstein order trades master-process inflation against the number of forward-oracle calls used by each factory event. Proposition~\ref{prop:bn} shows that the compensation is of order $n^{-1}$. The next subsection specializes to a common small-noise scale and common order in order to derive both a full-batch benchmark and the sharp work of the actual range-stopped implementation.

Combining the Bernstein likelihood coin, the range-stopped implementation,
and Poisson thinning gives the complete exact sampler summarized in
Algorithm~\ref{alg:main}. For independent Gaussian observations its output
is the posterior point process
$\eta_\gamma\sim\PPP(\gamma\nu^y)$.

\begin{algorithm}[t]
\caption{Likelihood-coin posterior PPP for independent Gaussian observations}
\label{alg:main}
\begin{algorithmic}[1]
\State Choose orders $n_j$ and compute $b_{n_j}(y_j)$ and $B_{\bm n}(y)=\sum_j a_jb_{n_j}(y_j)$.
\State Set $A_{\bm n}:=\sum_{j=1}^J a_jC_{n_j}(y_j)$ and $\pi_j:=a_jC_{n_j}(y_j)/A_{\bm n}$.
\State Draw $M_\gamma\sim\Pois(\gamma e^{B_{\bm n}(y)})$ and candidates $Q_i\stackrel{\rm iid}{\sim}\mu_0$, $i=1,\ldots,M_\gamma$.
\For{$i=1,\ldots,M_\gamma$}
 \State Draw $R\sim\Pois(A_{\bm n})$ and set $\mathrm{keep}=1$.
 \For{$h=1,\ldots,R$}
  \State Independently draw a component label $j$ with probability $\pi_j$.
  \State Generate an early-stopped Bernstein event $H_{n_j,y_j}$ from stochastic $p_j(Q_i)$-coins.
  \If{$H_{n_j,y_j}=1$}
   \State Set $\mathrm{keep}=0$ and terminate the remaining likelihood tests for candidate $i$.
  \EndIf
 \EndFor
 \If{$\mathrm{keep}=1$} \State retain $Q_i$. \EndIf
\EndFor
\State Return the retained point process.
\end{algorithmic}
\end{algorithm}

The marked Poisson stream in Algorithm~\ref{alg:main} is equivalent, before
stopping, to independent component streams with counts
$R_j\sim\Pois(a_j C_{n_j}(y_j))$. Its randomized component order preserves
the likelihood-coin law and is used in the exact work identity below.
Algorithm~\ref{alg:main} contains two distinct forms of early termination.
Within each Bernstein event, the range test of Lemma~\ref{lem:early} may decide
the event before all \(n_j\) forward coins are generated. At the candidate
level, the first rejected auxiliary event terminates all remaining likelihood
tests for that candidate. Neither mechanism changes the retained point-process
law; both affect only the amount of forward-oracle work. The choice of the
orders \(n_j\), and hence the balance between master-process inflation and
range-stopped work, is analyzed in Section~\ref{subsec:sharp-work}.

\subsection{Common-order benchmark and sharp stopped-work complexity}
\label{subsec:sharp-work}

Algorithm~\ref{alg:main} saves work in two different ways. Within a single
Bernstein event, the range rule may determine the event before all forward
coins are revealed. At the candidate level, the first rejecting auxiliary
event terminates all remaining likelihood tests for that candidate. The goal
of this subsection is to quantify these two savings and to identify the
small-noise regime in which each one matters.

To keep the complexity analysis focused on the stopping mechanism rather than
on componentwise allocation, we specialize to a common observational scale
$\sigma_1=\cdots=\sigma_J=\sigma$, write
$a_\sigma=(2\sigma^2)^{-1}$, and use the same Bernstein order $n$ for every
component. The scaled order $\vartheta_\sigma=n/a_\sigma$ compares the Bernstein order
with the inverse-noise scale. We also set
\[
Q_y(q):=\sum_{j=1}^J[p_j(q)-y_j]^2,\qquad
b_n^{\mathrm{tot}}(y):=\sum_{j=1}^J b_n(y_j),\qquad
D_y:=\sum_{j=1}^J y_j(1-y_j).
\]
Here $Q_y(q)$ is the predictive squared misfit. The quantity
$b_n^{\mathrm{tot}}(y)$ is the total deterministic compensation introduced by the
Bernstein factories, and therefore controls the inflation of the master PPP.
By Proposition~\ref{prop:bn},
\begin{equation}
 b_n^{\mathrm{tot}}(y)=\frac{D_y}{n}+O(n^{-2}).
 \label{eq:sum-b-asymptotic}
\end{equation}

We begin with a benchmark in which both stopping mechanisms are disabled. In
this full-batch implementation every auxiliary event uses all $n$ forward
coins, and every Poisson event is processed even if the candidate could
already have been rejected.

\begin{proposition}
\label{prop:full-batch-common}
Assume $y\in(0,1)^J$ and define the scaled Bernstein order
$\vartheta_\sigma:=n/a_\sigma$. Suppose that
$\vartheta_\sigma\to\vartheta\in(0,\infty)$ as $\sigma\downarrow0$.
Equivalently, $n=n(\sigma)\sim\vartheta a_\sigma$, so $\vartheta$
measures the asymptotic Bernstein order per unit small-noise scale. If
$W_{\rm full}$ denotes the total number of forward Bernoulli calls in the
full-batch implementation, then
\[
\frac{\E W_{\rm full}}{\gamma a_\sigma^2}
 \longrightarrow J\vartheta\exp\!\left(\frac{D_y}{\vartheta}\right).
\]
Among orders with $n/a_\sigma\to\vartheta\in(0,\infty)$, the limit is
minimized at $\vartheta=D_y$. Hence the minimizing full-batch scale satisfies
\[
n\sim a_\sigma D_y,\qquad
a_\sigma b_n^{\mathrm{tot}}(y)\to1,
\qquad e^{a_\sigma b_n^{\mathrm{tot}}(y)}\to e.
\]
\end{proposition}

\begin{proof}
The master PPP has mean size $\gamma e^{a_\sigma b_n^{\mathrm{tot}}(y)}$.
Because $C_n(y_j)=1$, each of the $J$ component factories has a
$\Pois(a_\sigma)$ auxiliary count, and each completed event costs exactly $n$
forward calls. Therefore
$\E W_{\rm full}
 =\gamma e^{a_\sigma b_n^{\mathrm{tot}}(y)}J a_\sigma n$.
Equation~\eqref{eq:sum-b-asymptotic} and $\vartheta_\sigma\to\vartheta$ give the stated
limit. Minimizing $\vartheta e^{D_y/\vartheta}$ yields $\vartheta=D_y$.
\end{proof}

Proposition~\ref{prop:full-batch-common} identifies the full-batch benchmark:
without early stopping, the natural complexity scale is
$\gamma a_\sigma^2=\Theta(\gamma\sigma^{-4})$. We next quantify the first
saving mechanism: stopping within a single Bernstein event.

For a scalar component with forward probability $p\in[0,1]$, let
$\tau_{n,y}(p)$ be the number of $p$-coin calls used by the range rule in
Lemma~\ref{lem:early}, and let $\bar\tau_{n,y}(p)=\E\tau_{n,y}(p)$. Suppose
that a fraction $t$ of the $n$ forward coins has been revealed. For large $n$, the revealed successes contribute approximately
$tp$ to the final normalized count, while the unrevealed fraction $1-t$ may
still contribute anything between $0$ and $1-t$. This motivates the interval
$I_t(p):=[tp,tp+1-t]$. Define
\[
\kappa_y(p):=\int_0^1\left(
\sup_{z\in I_t(p)}(z-y)^2-
\inf_{z\in I_t(p)}(z-y)^2\right)\,\mathrm dt.
\]
The integrand measures how much uncertainty about the final normalized count
can still change the squared residual after a fraction $t$ of the forward
coins has been revealed. Accordingly, $\kappa_y(p)$ will be the asymptotic
fraction of the $n$ calls that remain necessary under the range rule.

\begin{proposition}
\label{prop:scalar-range-cost}
For every fixed $y\in[0,1]$,
\begin{equation}
 \sup_{p\in[0,1]}
 \left|\frac{\bar\tau_{n,y}(p)}{n}-\kappa_y(p)\right|\longrightarrow0.
 \label{eq:scalar-event-cost-limit}
\end{equation}
Moreover, $\kappa_y$ is continuous and $1/12\le\kappa_y(p)\le1$ for all
$p\in[0,1]$.
\end{proposition}

\begin{proof}
After $m$ forward calls with $K_m$ successes, the final normalized success
count can only lie in
\[
\left[\frac{K_m}{n},\frac{K_m+n-m}{n}\right].
\]
A maximal inequality for centered Bernoulli partial sums gives
$\sup_{p\in[0,1]}\E\max_{m\le n}|K_m-mp|/n=O(n^{-1/2})$.
Thus the reachable interval approaches $I_{m/n}(p)$ uniformly in $p$
and normalized time. The normalized threshold
function is $(z-y)^2+o(1)$ uniformly on $[0,1]$; the deterministic
compensation is only an additive shift and therefore does not change the
width of its reachable range. Conditional on the currently revealed coins,
the auxiliary uniform mark remains unresolved precisely when it falls inside
that threshold range. The tail-sum formula for $\tau_{n,y}(p)$ and a
Riemann-sum argument then give \eqref{eq:scalar-event-cost-limit}. Continuity
follows from the definition. Since $(z-y)^2\in[0,1]$, the upper bound is
immediate. Any interval of length $1-t$ has squared-residual range at least
$(1-t)^2/4$; integrating over $t\in[0,1]$ gives the lower bound $1/12$.
\end{proof}

For the vector likelihood write
\[
\bar\tau_n^{\mathrm{tot}}(q):=
\sum_{j=1}^J\bar\tau_{n,y_j}(p_j(q)),\qquad
\kappa_y^{\mathrm{tot}}(p):=
\sum_{j=1}^J\kappa_{y_j}(p_j).
\]
The previous proposition describes the cost of one auxiliary event. The
second, and more important, saving comes from terminating a candidate as soon
as the first auxiliary event rejects it. The next identity combines both
levels of stopping.

\begin{theorem}[Exact mean actual-work identity]
\label{thm:actual-work-identity}
Let $W_{\rm act}$ be the total number of forward Bernoulli calls made by
Algorithm~\ref{alg:main} with common order $n$, exact scalar range stopping,
and candidate termination at the first auxiliary event equal to one. Then
\begin{equation}
 \E W_{\rm act}
 =\gamma e^{a_\sigma b_n^{\mathrm{tot}}(y)}
 \int_{\X}\bar\tau_n^{\mathrm{tot}}(q)
 \frac{1-e^{-a_\sigma[b_n^{\mathrm{tot}}(y)+Q_y(q)]}}
 {b_n^{\mathrm{tot}}(y)+Q_y(q)}\,\mu_0(\mathrm dq).
 \label{eq:actual-work-identity}
\end{equation}
\end{theorem}

\begin{proof}
In the common-order setting $C_n(y_j)=1$, so the marked stream in
Algorithm~\ref{alg:main} has $R\sim\Pois(Ja_\sigma)$ events and independent
uniform component labels in $\{1,\ldots,J\}$. For fixed $q$,
a randomly labelled event rejects with probability
$h_n(q)=[b_n^{\mathrm{tot}}(y)+Q_y(q)]/J$, while its mean forward-call cost is
$J^{-1}\bar\tau_n^{\mathrm{tot}}(q)$. Conditional on $R=r$, the expected number of
events that are started before the first rejection is
$[1-(1-h_n(q))^r]/h_n(q)$. Averaging over the Poisson count gives
\[
J\,\frac{1-e^{-a_\sigma[b_n^{\mathrm{tot}}(y)+Q_y(q)]}}
 {b_n^{\mathrm{tot}}(y)+Q_y(q)}.
\]
Multiplying by the mean cost of a started event cancels the factor $J$.
Finally, averaging over the master PPP gives
\eqref{eq:actual-work-identity}.
\end{proof}

The identity separates the two sources of computational cost and makes the small-noise mechanism transparent. Apart from the
master-process inflation and the cost of one started event, the decisive factor
is
\[
g_{a_\sigma}(s):=\frac{1-e^{-a_\sigma s}}{s},\qquad s>0.
\]
If $a_\sigma s\ll1$, then $g_{a_\sigma}(s)\sim a_\sigma$; if
$a_\sigma s\gg1$, then $g_{a_\sigma}(s)\sim s^{-1}$. Thus a parameter value
with large predictive misfit $Q_y(q)$ is rejected after only a few auxiliary
events, whereas a nearly exact fit with $Q_y(q)\approx0$ survives much longer.
The total work is therefore determined by how much prior mass lies near the
exact-fit set $\{q:Q_y(q)=0\}$.

We encode this local geometry through a small-ball condition.

\begin{assumption}[Prior-predictive small-ball law]
\label{ass:small-ball}
For the fixed observation vector $y$, there exist $c_y>0$ and $\alpha>0$ such
that
\[
F_y(t):=\mu_0\{q:Q_y(q)\le t\}
 =c_y t^\alpha\{1+o(1)\},\qquad t\downarrow0.
\]
\end{assumption}

The exponent $\alpha$ measures the local amount of prior mass that can produce
an almost exact fit. If the prior predictive law of $p(q)$ has a
nondegenerate $d_{\mathrm{pred}}$-dimensional density near $y$, then
$Q_y(q)=\|p(q)-y\|_2^2$ and a ball of squared radius $t$ has mass of order
$t^{d_{\mathrm{pred}}/2}$. Hence $\alpha=d_{\mathrm{pred}}/2$, so
$d_{\mathrm{pred}}=2\alpha$ is the effective local predictive dimension. This dimension concerns the image of
the prior under the observation map and need not agree with the dimension of
the unknown $q$.

For $0<\alpha<1$, define
\begin{equation}
 \Psi_\alpha(\beta)
 :=\alpha\int_0^\infty t^{\alpha-1}
 \frac{1-e^{-(\beta+t)}}{\beta+t}\,\mathrm dt
 =\Gamma(\alpha+1)\int_0^1e^{-\beta s}s^{-\alpha}\,\mathrm ds.
 \label{eq:Psi-alpha}
\end{equation}

\begin{theorem}[Sharp small-noise complexity]
\label{thm:sharp-actual-work}
Assume $y\in(0,1)^J$, Assumption~\ref{ass:small-ball}, and
$\vartheta_\sigma=n/a_\sigma\to\vartheta>0$. Put $\beta=D_y/\vartheta$. Then
\begin{equation}
 \E W_{\rm act}\sim
 \begin{cases}
 \gamma e^\beta \vartheta c_y\kappa_y^{\mathrm{tot}}(y)
 \Psi_\alpha(\beta)a_\sigma^{2-\alpha},&0<\alpha<1,\\[1.1ex]
 \gamma e^\beta \vartheta c_y\kappa_y^{\mathrm{tot}}(y)
 a_\sigma\log a_\sigma,&\alpha=1,\\[1.1ex]
 \gamma e^\beta \vartheta\mathcal I_y a_\sigma,&\alpha>1,
 \end{cases}
 \label{eq:sharp-work-rates}
\end{equation}
where
$\mathcal I_y:=
\int_{\X}\frac{\kappa_y^{\mathrm{tot}}(p(q))}{Q_y(q)}\,\mu_0(\mathrm dq)<\infty$, $(\alpha>1)$.
Consequently,
\[
\E W_{\rm act}=
 \begin{cases}
 \Theta(\gamma\sigma^{-4+2\alpha}),&0<\alpha<1,\\
 \Theta(\gamma\sigma^{-2}\log(1/\sigma)),&\alpha=1,\\
 \Theta(\gamma\sigma^{-2}),&\alpha>1.
 \end{cases}
\]
\end{theorem}

\begin{proof}
Equation~\eqref{eq:sum-b-asymptotic} gives
$a_\sigma b_n^{\mathrm{tot}}(y)\to\beta$, while
Proposition~\ref{prop:scalar-range-cost} gives uniformly
\[
\bar\tau_n^{\mathrm{tot}}(q)
=a_\sigma\vartheta\{\kappa_y^{\mathrm{tot}}(p(q))+o(1)\}.
\]
Near the exact-fit set, $Q_y(q)\to0$ implies $p(q)\to y$, so the event-level
cost converges to $a_\sigma\vartheta\kappa_y^{\mathrm{tot}}(y)$. Applying
the regular-variation lemma in Appendix~\ref{app:small-ball} to the remaining factor in
\eqref{eq:actual-work-identity} yields the three regimes. Finally,
$a_\sigma=(2\sigma^2)^{-1}$.
\end{proof}

The theorem shows precisely what early candidate rejection changes. The
full-batch benchmark is always of order $\gamma\sigma^{-4}$. With exact
stopping, the rate is governed by the local predictive exponent. In
particular, if $d_{\mathrm{pred}}=2\alpha$, then
\[
d_{\mathrm{pred}}=1:\ \Theta(\gamma\sigma^{-3}),\qquad
d_{\mathrm{pred}}=2:\ \Theta(\gamma\sigma^{-2}\log(1/\sigma)),\qquad
d_{\mathrm{pred}}>2:\ \Theta(\gamma\sigma^{-2}).
\]
The transition at $d_{\mathrm{pred}}=2$ comes from the borderline integrability of
$Q_y(q)^{-1}$ near the exact-fit set.

For fixed $\gamma$, however, the expected number of retained posterior atoms
also decreases as the noise level shrinks. Indeed, the same small-ball law
gives
\[
Z_\sigma(y):=\int e^{-a_\sigma Q_y(q)}\,\mu_0(\mathrm dq)
 \sim c_y\Gamma(\alpha+1)a_\sigma^{-\alpha},
\]
so $\E N_\gamma=\gamma Z_\sigma(y)$. Dividing the total work by the expected
retained count gives the cost per posterior atom.

\begin{corollary}
\label{cor:work-per-atom}
Under Assumption~\ref{ass:small-ball},
\[
\frac{\E W_{\rm act}}{\E N_\gamma}=
 \begin{cases}
 \Theta(a_\sigma^2),&0<\alpha<1,\\
 \Theta(a_\sigma^2\log a_\sigma),&\alpha=1,\\
 \Theta(a_\sigma^{\alpha+1}),&\alpha>1.
 \end{cases}
\]
Equivalently, in terms of $d_{\mathrm{pred}}=2\alpha$, the cost per expected posterior atom is
$\Theta(\sigma^{-4})$ for $d_{\mathrm{pred}}=1$,
$\Theta(\sigma^{-4}\log(1/\sigma))$ for $d_{\mathrm{pred}}=2$, and
$\Theta(\sigma^{-(d_{\mathrm{pred}}+2)})$ for $d_{\mathrm{pred}}>2$.
\end{corollary}

\begin{proof}
Use the same regular-variation lemma for $Z_\sigma(y)$ and divide
\eqref{eq:sharp-work-rates} by $\gamma Z_\sigma(y)$.
\end{proof}

The same asymptotics determine how the Bernstein order should scale with the noise level.
Because $\vartheta_\sigma=n/a_\sigma\to\vartheta$ and $\beta=D_y/\vartheta$, the parameter $\beta$ is exactly
the limiting compensation exponent: $a_\sigma b_n^{\mathrm{tot}}(y)\to\beta$ and the
master-process inflation tends to $e^\beta$. Optimizing the order is therefore
equivalent to optimizing the leading constant with respect to $\beta$.

\begin{theorem}
\label{thm:actual-optimal-order}
Among orders with $n/a_\sigma\to\vartheta\in(0,\infty)$, under the
assumptions of Theorem~\ref{thm:sharp-actual-work}, the leading work
is minimized by
$n^*(\sigma)\sim a_\sigma D_y/\beta_\alpha$, where
\[
\beta_\alpha=
 \begin{cases}
 \displaystyle\arg\min_{\beta>0}
 \frac{e^\beta\Psi_\alpha(\beta)}{\beta},&0<\alpha<1,\\[1.8ex]
 1,&\alpha\ge1.
 \end{cases}
\]
For $0<\alpha<1$, the minimizer is unique and satisfies
$1-\beta_\alpha^{-1}
 +\Psi_\alpha'(\beta_\alpha)/\Psi_\alpha(\beta_\alpha)=0$.
In particular, $\beta_{1/2}=1.304366831\ldots$, so in predictive dimension
one the optimal early-stopped order is about $0.7667$ times the full-batch
order.
\end{theorem}

\begin{proof}
Since $\vartheta=D_y/\beta$, the $\beta$-dependent factor in
\eqref{eq:sharp-work-rates} is
$D_y e^\beta\Psi_\alpha(\beta)/\beta$ for $0<\alpha<1$ and
$D_y e^\beta/\beta$ for $\alpha\ge1$. The first objective is strictly
log-convex because $\Psi_\alpha$ is a Laplace transform of a positive measure;
the second is minimized at $\beta=1$. Numerical solution at $\alpha=1/2$
gives the stated value.
\end{proof}

Thus early stopping does more than reduce a constant. The within-event range rule changes the average
cost of each started Bernstein event, while candidate-level termination changes
the small-noise exponent itself. The latter effect is controlled by the local
prior-predictive geometry through $\alpha$, not by the nominal dimension of
$q$. The only remaining model-specific requirement is exact Bernoulli access
to the bounded forward observables; Section~\ref{sec:pde} supplies such access
for a class of elliptic PDEs.

\section{A PDE realization of the Bernoulli-access interface}
\label{sec:pde}

The likelihood-coin framework developed in Section~\ref{sec:factory} is
model-agnostic once exact Bernoulli access to the forward observables is
available. The likelihood factory does not otherwise depend on the governing equation. This section provides one concrete realization of that interface for a bounded class of elliptic
PDEs. A Feynman--Kac representation expresses each observable as the
expectation of a bounded path functional, which can be realized directly as a
Bernoulli event rather than estimated by Monte Carlo averaging. For the
elliptic resolvent problem below, the reaction term is implemented through
Poisson killing and the source term through a terminal Bernoulli mark.
Function-valued coefficients are handled by lazy evaluation of convergent
random series, so the posterior sampler requires neither a deterministic PDE
solve nor a fixed truncation of the unknown.

\subsection{Feynman--Kac representation and Bernoulli access}

Let $D\subset\mathbb R^d$ be bounded, let $(\mathsf B_t)_{t\ge0}$ be Brownian
motion with generator $\frac12\Delta$, and let $\tau_D$ be its first exit time.
For
\[
\left(q-\frac12\Delta\right)u=f \quad\hbox{in }D,
\qquad u=g \quad\hbox{on }\partial D,
\]
with $q\ge0$, the elliptic Feynman--Kac formula gives
\[
 u(x)=\E_x\!\left[e^{-\int_0^{\tau_D}q(\mathsf B_s)\,\mathrm ds}
 g(\mathsf B_{\tau_D})+
 \int_0^{\tau_D}e^{-\int_0^tq(\mathsf B_s)\,\mathrm ds}f(\mathsf B_t)\,\mathrm dt\right].
\]
See, for example, \cite{Evans2013SDE,KaratzasShreve1991}. The exponential
factor is the survival probability under the pathwise hazard $q(\mathsf B_t)$.
Whenever a stochastic representation can be written as
$u_q(x)=\E[Y_q(x)\mid q]$ with $0\le Y_q(x)\le1$, an independent
$U\sim\Unif(0,1)$ turns it directly into the exact forward coin
$\ind_{\{U<Y_q(x)\}}\sim\Bern(u_q(x))$. Thus the method uses a single
randomized path event rather than first estimating $u_q(x)$ numerically:
\[
\text{path event}\longrightarrow u_q(x)\text{-coin}
\longrightarrow\text{likelihood coin}\longrightarrow\text{PPP thinning}.
\]

\subsection{A bounded Feynman--Kac resolvent class}

On the two-dimensional torus $\Torus$, let $\lambda>0$ and $q_{\max}>0$, and assume $0\le q(x)\le q_{\max}$, $0\le h(x)\le\lambda$.
Consider
\begin{equation}
 \left(\lambda+q(x)-\frac12\Delta\right)u(x)=h(x)
 \qquad\hbox{on }\Torus.
 \label{eq:bounded-resolvent-pde}
\end{equation}
If $T\sim\operatorname{Exp}(\lambda)$ is independent of Brownian motion on $\Torus$, the resolvent form of Feynman--Kac is
\begin{equation}
 u(s)=
 \E_s\!\left[
 \exp\!\left\{-\int_0^T q(\mathsf B_t)\dd t\right\}
 \frac{h(\mathsf B_T)}{\lambda}
 \right].
 \label{eq:bounded-resolvent-fk}
\end{equation}
The bounds on $q$ and $h$ imply $0\le u\le1$. Here $\lambda$ is both the resolvent parameter and the rate of the exponential clock; the choice $h\equiv\lambda$ used later is only a specialization for which the terminal source coin is identically one.

Assume that at a queried location $x$ we can generate independent coins with success probabilities $q(x)/q_{\max}$ and $h(x)/\lambda$. The first coin marks potential killing events and the second supplies the terminal payoff.

\begin{theorem}[Poisson-killing forward coin]
\label{thm:reaction-coin}
Fix $s\in\Torus$. Draw $T\sim\operatorname{Exp}(\lambda)$ and generate a homogeneous Poisson process of rate $q_{\max}$ on $[0,T]$. At each event time $\tau_r$, simulate the Brownian location $\mathsf B_{\tau_r}$ and generate a $q(\mathsf B_{\tau_r})/q_{\max}$-coin. Return zero at the first successful killing mark. If no killing mark succeeds before $T$, generate an $h(\mathsf B_T)/\lambda$-coin and return its outcome. Then
\[
J_u(s)\sim\Bern(u(s)),
\]
where $u$ is given by \eqref{eq:bounded-resolvent-fk}.
\end{theorem}

\begin{proof}
Condition on $T$ and the Brownian path. Thinning the rate-$q_{\max}$ Poisson process with acceptance probability $q(\mathsf B_t)/q_{\max}$ gives an inhomogeneous killing process with rate $q(\mathsf B_t)$. The conditional probability of no accepted killing event before $T$ is therefore
 $\exp\!\left\{-\int_0^Tq(\mathsf B_t)\dd t\right\}$.
Conditional on survival, the terminal coin succeeds with probability $h(\mathsf B_T)/\lambda$. Taking expectation with respect to the Brownian path and the independent exponential time $T$ gives \eqref{eq:bounded-resolvent-fk}.
\end{proof}

\subsection{Lazy Bernoulli access to function-valued unknowns}

Let $\{\phi_\ell\}_{\ell\ge1}$ be measurable functions on a domain $\mathcal D$ with $|\phi_\ell(x)|\le1$, and let
\[
f_\xi(x)
 =b+\sum_{\ell=1}^\infty \omega_\ell\xi_\ell\phi_\ell(x),
 \qquad |\xi_\ell|\le1,
\]
where
\begin{equation}
 A:=\sum_{\ell=1}^\infty|\omega_\ell|<\min\{b,1-b\}.
 \label{eq:field-bound}
\end{equation}
The coordinate space is $\mathcal X_\xi=[-1,1]^{\mathbb N}$ with its product $\sigma$-algebra and a product prior $\pi_0$. If the $\phi_\ell$ are continuous, absolute summability implies uniform convergence and measurability of $\xi\mapsto f_\xi$ into $C(\mathcal D)$; the coordinate prior therefore induces a function-space prior by pushforward.

By \eqref{eq:field-bound}, $f_\xi(x)\in(0,1)$ uniformly. Define
$S_r(x)=b+\sum_{\ell=1}^r \omega_\ell\xi_\ell\phi_\ell(x)$,
 $R_r=\sum_{\ell>r}|\omega_\ell|$.
Then
\begin{equation}
 S_r(x)-R_r\le f_\xi(x)\le S_r(x)+R_r.
 \label{eq:tail-interval}
\end{equation}
For a single $U\sim\Unif(0,1)$, reveal coefficients sequentially until either $U<S_r-R_r$ or $U>S_r+R_r$.

\begin{theorem}[Exact lazy function-value coin]
\label{thm:lazy-function}
Assume \eqref{eq:field-bound}. For every fixed $x$, the lazy decision stops after finitely many coefficient reveals almost surely and returns
$J_f(x)=\ind_{\{U<f_\xi(x)\}}$.
Hence, conditional on $f_\xi$,
$J_f(x)\sim\Bern(f_\xi(x))$.
If $\tau$ is the number of revealed coefficients, then
\begin{equation}
 \Pp(\tau>r\mid\xi)\le \min\{1,2R_r\},
 \qquad
 \E(\tau\mid\xi)\le\sum_{r=0}^\infty\min\{1,2R_r\}.
 \label{eq:tau-bound}
\end{equation}
In particular, if
$|\omega_\ell|=A(1-\rho)\rho^{\ell-1}$, $0<\rho<1$,
then $R_r=A\rho^r$ and, when $2A\le1$,
$\E\tau\le\frac{2A}{1-\rho}$.
\end{theorem}

\begin{proof}
Since $R_r\to0$, the intervals in \eqref{eq:tail-interval} shrink to $f_\xi(x)$. Because $U$ is continuous and independent, $\Pp(U=f_\xi(x)\mid\xi)=0$, so the decision resolves almost surely. The event $\{\tau>r\}$ requires $U\in[S_r-R_r,S_r+R_r]\cap[0,1]$, whose length is at most $2R_r$. Summing the tail probabilities gives \eqref{eq:tau-bound}; the geometric case follows directly.
\end{proof}

A retained atom may therefore be represented by a lazily revealed infinite coefficient sequence. Each resolved field-value decision is measurable with respect to the revealed prefix and its auxiliary uniform variable and is unchanged for every admissible completion of the tail. Because the coordinate prior is a product measure, coefficients beyond the largest revealed index retain their product-prior law conditional on the realized algorithmic history and may be generated later when a fuller realization is needed.

We use two specializations of \eqref{eq:bounded-resolvent-pde}. If $q\equiv0$ and $h=\lambda f_\xi$, then
\[
\left(\lambda-\frac12\Delta\right)u_{f_\xi}=\lambda f_\xi,
 \qquad
 u_{f_\xi}(s)=\E_s[f_\xi(\mathsf B_T)],
\]
and Theorem~\ref{thm:lazy-function} supplies the terminal $f_\xi(\mathsf B_T)$-coin. This gives an exact inverse-source specialization of the same oracle construction.

If $h\equiv\lambda$ and
\begin{equation}
 q_\xi(x)=q_{\max}f_\xi(x),
 \qquad 0<f_\xi(x)<1,
 \label{eq:qfield}
\end{equation}
then
\begin{equation}
 \left(\lambda+q_\xi(x)-\frac12\Delta\right)u_{q_\xi}=\lambda,
 \label{eq:reaction-function-pde}
\end{equation}
and
\[
p_s(q_\xi):=u_{q_\xi}(s)
 =\E_s\!\left[\exp\!\left\{-\int_0^Tq_\xi(\mathsf B_t)\dd t\right\}\right].
\]
Here the terminal coin is one and each killing mark is generated by the lazy $f_\xi$-coin.

For this reaction specialization, define
$\mathcal R_f:=\sum_{r=0}^{\infty}\min\{1, 2R_r\}$.
For a common Bernstein order \(n\), write
$B_n(y):=\sum_{j=1}^J a_j b_n(y_j)$.
\begin{proposition}
\label{prop:reaction-work}
Let $N_{\rm kill}$ be the number of candidate killing events inspected by one forward coin for \eqref{eq:reaction-function-pde}, and let $N_{\rm coef}$ be the total number of random-series coefficient terms inspected by its lazy field-value coins, counting repeated inspections across distinct field queries. Then
\[
\E N_{\rm kill}\le \frac{q_{\max}}{\lambda},
 \qquad
 \E N_{\rm coef}\le \frac{q_{\max}}{\lambda}\,\mathcal R_f.
\]
For a common Bernstein order $n$ and Gaussian observations \eqref{eq:multi-L},
\begin{equation}
 \E W_{\rm coef}
 \le
 \gamma e^{B_n(y)}
 n\left(\sum_{j=1}^J a_jC_n(y_j)\right)
 \frac{q_{\max}}{\lambda}\,\mathcal R_f.
 \label{eq:total-work-bound}
\end{equation}
The bound contains neither a spatial discretization parameter nor a fixed truncation dimension for the random field.
\end{proposition}

\begin{proof}
If every rate-$q_{\max}$ event on $[0,T]$ were inspected, the conditional mean count would be $q_{\max}T$. Early termination at the first accepted killing event can only reduce the count, so $\E N_{\rm kill}\le q_{\max}\E T=q_{\max}/\lambda$. Each inspected killing location invokes a lazy field-value coin with expected inspection count at most $\mathcal R_f$. A Bernstein event uses at most $n$ forward coins, sensor $j$ generates a Poisson number of such events with mean $a_jC_n(y_j)$, and the inflated master PPP has mean cardinality $\gamma e^{B_n(y)}$. Ignoring all early terminations gives \eqref{eq:total-work-bound}.
\end{proof}

Proposition~\ref{prop:reaction-work} is a model-side work bound for one Bernoulli-accessible PDE class. When the prior predictive law of the observable vector also satisfies Assumption~\ref{ass:small-ball}, Theorem~\ref{thm:sharp-actual-work} supplies the corresponding sharp likelihood-factory rate; the factor $q_{\max}\mathcal R_f/\lambda$ remains a conservative bound for the average coefficient-query cost of one forward Bernoulli call.

\begin{corollary}
\label{cor:reaction-posterior}
Let $\pi_0$ be a product prior on $\mathcal X_\xi$, let $\mathcal Q(\xi)=q_\xi$, and let $\mu_0=\mathcal Q_\#\pi_0$ be the induced prior on the reaction fields \eqref{eq:qfield}. Use Theorem~\ref{thm:reaction-coin} to generate the forward coins in the Bernstein--Poisson likelihood factory at every sensor. Then
\begin{equation}
 \eta_\gamma\sim\PPP(\gamma\nu^y).
 \label{eq:reaction-posterior-ppp}
\end{equation}
Conditional on its cardinality, the retained fields are independent draws from the continuum Bayesian posterior. No deterministic spatial discretization or fixed parameter truncation enters the mathematical target.
\end{corollary}

\begin{proof}
Theorem~\ref{thm:reaction-coin} gives exact Bernoulli access to each $p_j(q_\xi)$. The Gaussian factory therefore supplies a coin with success probability $e^{-B_{\bm n}(y)}L(q_\xi;y)$, and Corollary~\ref{cor:gaussian-ppp} yields \eqref{eq:reaction-posterior-ppp}. The conditional-iid statement follows from Theorem~\ref{thm:coin-ppp}.
\end{proof}

The continuum oracle is deliberately restricted to bounded Bernoulli-accessible
resolvent observables. In the present construction the reaction term is
nonnegative and bounded by a known \(q_{\max}\), and the terminal source mark
lies in \([0,1]\). Unknown diffusion coefficients, sign-changing or unbounded
potentials, and oscillatory or complex-valued observables require different
stochastic representations and are outside the current theory.


\section{Numerical experiments}
\label{sec:numerics}

The experiments test the method in four stages: correctness and calibration in a two-parameter benchmark; sharp stopped-work scaling; a function-valued continuum PDE realization with lazy coefficient revelation; and a matched-accuracy comparison with finite-difference prior rejection. Deterministic PDE solves are used only for validation and comparison, not inside the continuum sampler.

\subsection{Two-parameter PDE benchmark}

We first use a two-dimensional problem with an analytically available reference to validate separately the forward Bernoulli oracle, the likelihood factories, and the resulting posterior Poisson thinning. 

On $\Torus$, let $\theta=(\theta_1,\theta_2)\in[-1,1]^2$ and consider
\[
 \left(36-\frac12\Delta\right)u_\theta(x)
 =32\left\{\frac12+0.24\theta_1\cos(2\pi x_1)+0.24\theta_2\cos(2\pi x_2)\right\}.
\]
The prior is uniform on $[-1,1]^2$. At $s_1=(0,1/4)$ and $s_2=(1/4,0)$,
\[
 p_j(\theta)=\frac{16}{36}+\frac{32(0.24)}{36+2\pi^2}\theta_j,
 \qquad j=1,2.
\]
The closed form is used only for deterministic reference calculations and calibration; the continuum PDE oracle uses the Feynman--Kac construction of Section~\ref{sec:pde}.

We take $\theta^\dagger=(0.55,-0.35)$, $\sigma=0.055$, and add the fixed perturbation $(0.012,-0.010)$, giving $y=(0.532226,0.386220)$. Table~\ref{tab:two-parameter-forward} compares exact probabilities with $30000$ continuum forward coins.

\begin{table}[!tbp]
\centering\small
\caption{Two-parameter benchmark: continuum forward-coin validation.}
\label{tab:two-parameter-forward}
\begin{tabular}{ccccc}
\toprule
sensor & exact forward & coin estimate & observed data & $2\widehat{\mathrm{SE}}$\\
\midrule
$(0,1/4)$ & 0.520226 & 0.521000 & 0.532226 & 0.005768\\
$(1/4,0)$ & 0.396220 & 0.396200 & 0.386220 & 0.005648\\
\bottomrule
\end{tabular}
\end{table}

For posterior sampling we use $n=128$ and $\gamma=3000$. Under independent errors the scalar compensations give master inflation $1.8786$. For the correlated experiment we take
 $\Sigma=\begin{pmatrix}1&0.75\\0.75&1\end{pmatrix}$,
and apply Theorem~\ref{thm:correlated-factory}. The discrete multivariate compensation is $b_{\bm n}^{\Omega}(y)=0.008717$, the auxiliary normalization is $C_{\bm n}^{\Omega}(y)=2.6373$, and the corresponding master inflation is $4.2241$. In both panels of Figure~\ref{fig:two-parameter-thinning}, retention is decided by the actual likelihood coin. For the correlated coin the reachable threshold range is propagated backward over the balanced component-reveal schedule, so the multivariate Bernstein event can also terminate before all $2n$ input coins are revealed.

\begin{figure}[!tbp]
\centering
\includegraphics[width=0.93\textwidth]{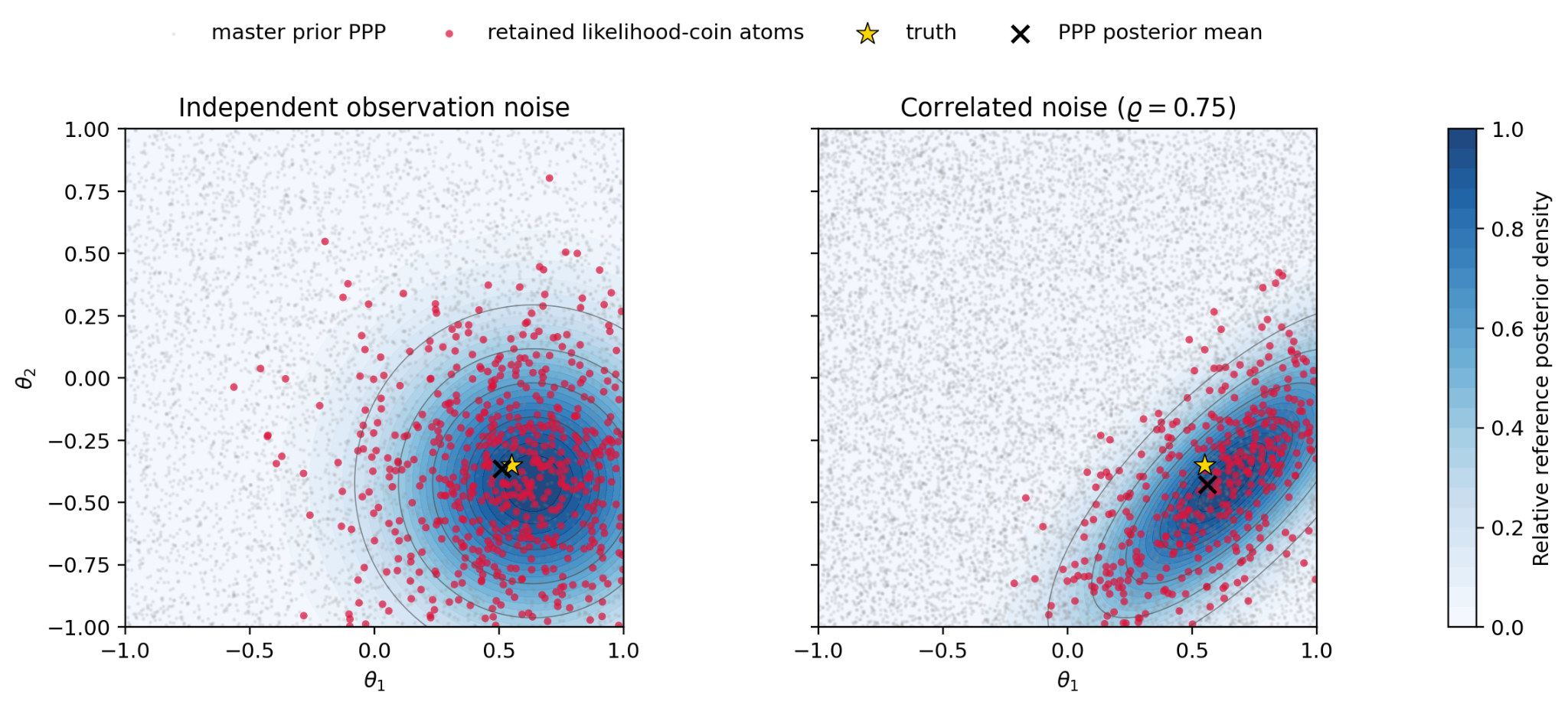}
\caption{Posterior PPPs for independent noise (left) and correlated noise with $\varrho=0.75$ (right). Gray points are master candidates, red points are retained atoms, and background contours show the deterministic reference posterior.}
\label{fig:two-parameter-thinning}
\end{figure}

For independent noise, the reference evidence is $0.18943$ and the displayed realization gives $N_\gamma/\gamma=0.19367$; the PPP posterior mean $(0.51030,-0.36354)$ is close to the quadrature value $(0.50891,-0.36276)$. For correlated noise, the corresponding values are $0.12313$ and $0.12600$, while the posterior means are $(0.55225,-0.45358)$ from quadrature and $(0.56071,-0.42746)$ from the PPP. The rotation of the right-hand posterior reflects the off-diagonal precision terms rather than a change in the forward model.

Figure~\ref{fig:two-parameter-calibration} provides a direct calibration
of both likelihood factories. At each of four fixed parameter values,
we run $10000$ independent likelihood coins and divide the raw success
frequency by the known compensation scale. The compensated frequencies
are plotted against the exact reduced likelihoods, which are available
analytically in this benchmark. In all eight experiments, the exact
likelihood lies within the reported two-standard-error Monte Carlo
interval around the compensated coin estimate. 

\begin{figure}[!tbp]
\centering
\includegraphics[width=0.93\textwidth]
{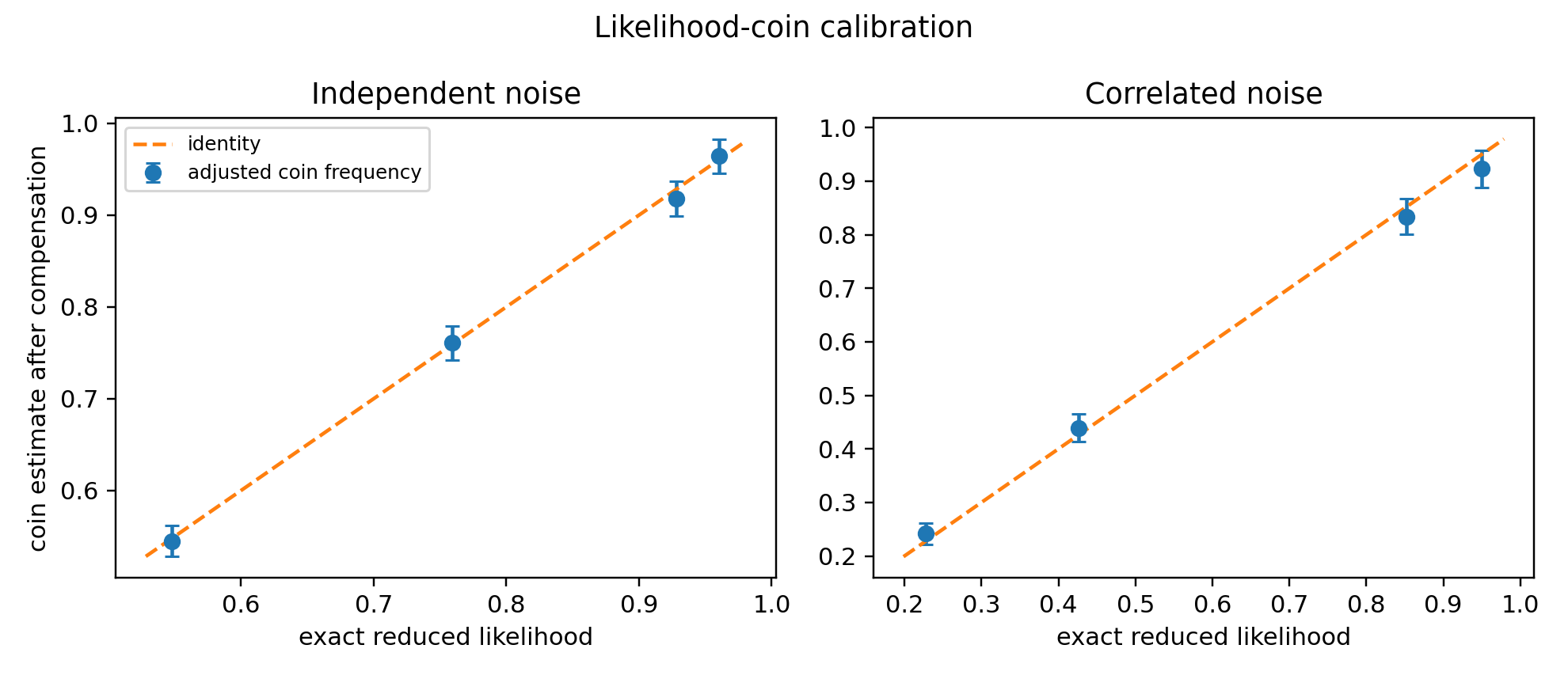}
\caption{Likelihood-coin calibration. Points are compensated coin frequencies with two-standard-error bars; the dashed line is the identity.}
\label{fig:two-parameter-calibration}
\end{figure}

Thus this benchmark isolates the probabilistic components of the method in a setting where the forward probabilities and posterior reference are independently available. The next experiment turns from correctness to the predicted work scaling, while the later function-valued example removes the availability of a finite-dimensional analytic forward representation.



\subsection{Small-ball dimension and sharp actual-work scaling}

For $d_{\mathrm{pred}}\in\{1,2,3\}$, take a uniform prior on $[0,1]^{d_{\mathrm{pred}}}$, the identity forward map, and observation vector $(1/2,\ldots,1/2)$ with independent errors of common variance $\sigma^2$. Then $\alpha=d_{\mathrm{pred}}/2$. The implementation uses the marked Poisson stream of Algorithm~\ref{alg:main},
as in Theorem~\ref{thm:actual-work-identity}: each event independently
selects a scalar component and uses its range-stopped Bernstein event; a
candidate terminates at the first rejection.

We use $\gamma=600$ and four independent runs at each noise level. The common order follows Theorem~\ref{thm:actual-optimal-order}; here $D_y=d_{\mathrm{pred}}/4$. Figure~\ref{fig:sharp-work} compares measured work with the sharp rates, and Table~\ref{tab:sharp-slopes} reports tail fits after removing the logarithmic factor for $d_{\mathrm{pred}}=2$.

\begin{table}[!tbp]
\centering
\caption{Sharp-complexity experiment: predicted rates and empirical tail exponents. Tail fits use the five largest values of $a_\sigma$. For $d_{\mathrm{pred}}=2$, the fitted quantity is $W/\log a_\sigma$.}
\label{tab:sharp-slopes}
\begin{tabular}{ccccc}
\toprule
$d_{\mathrm{pred}}$ & $\alpha$ & sharp scale & tail fitted exponent & predicted exponent\\
\midrule
1 & $1/2$ & $a_\sigma^{3/2}$ & 1.418 & 1.500\\
2 & 1 & $a_\sigma\log a_\sigma$ & 0.965 & 1.000\\
3 & $3/2$ & $a_\sigma$ & 1.006 & 1.000\\
\bottomrule
\end{tabular}
\end{table}

\begin{figure}[!tbp]
\centering
\begin{subfigure}{0.49\textwidth}
\centering
\includegraphics[width=\textwidth]{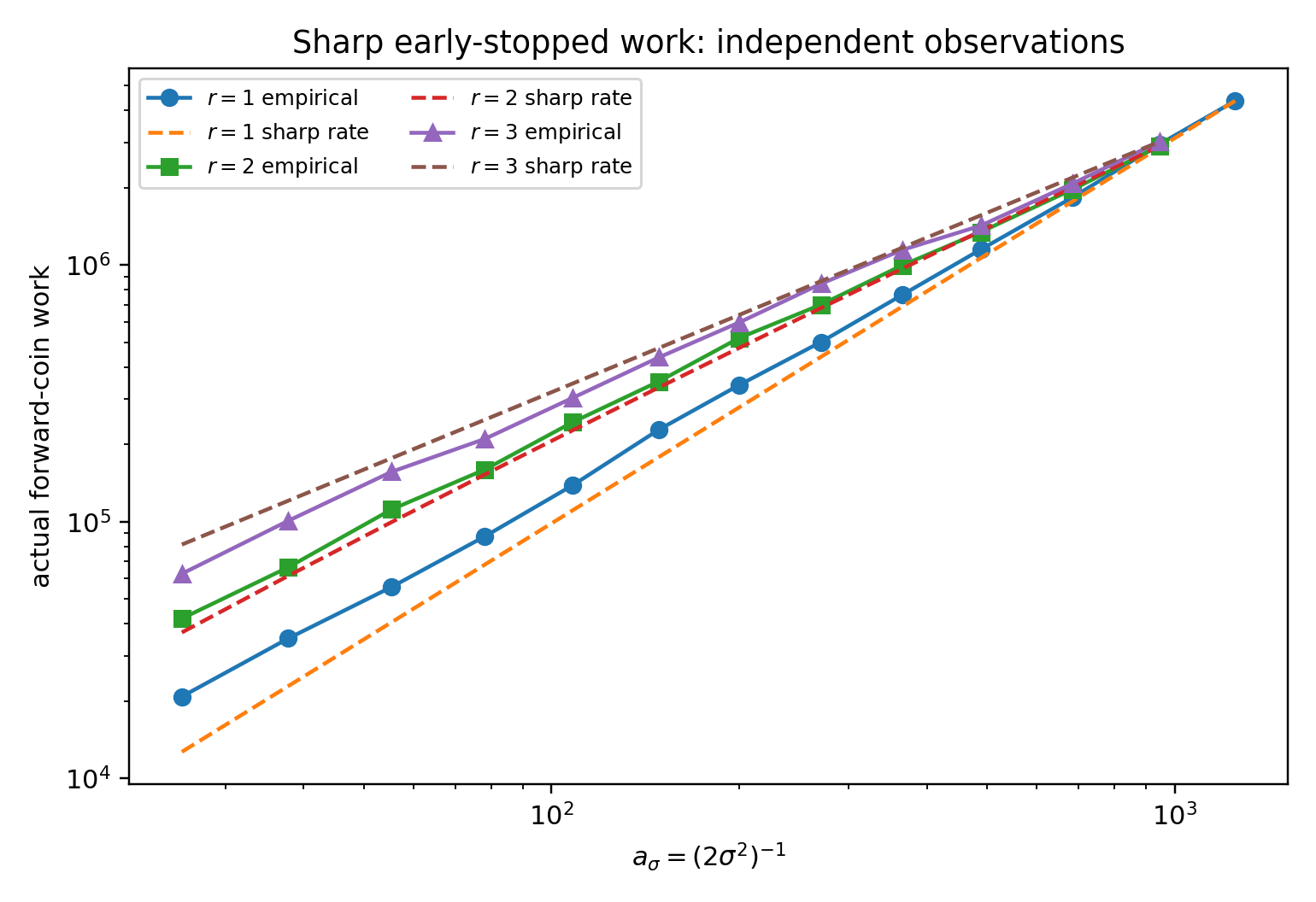}
\caption{Actual work and sharp reference rates.}
\end{subfigure}\hfill
\begin{subfigure}{0.49\textwidth}
\centering
\includegraphics[width=\textwidth]{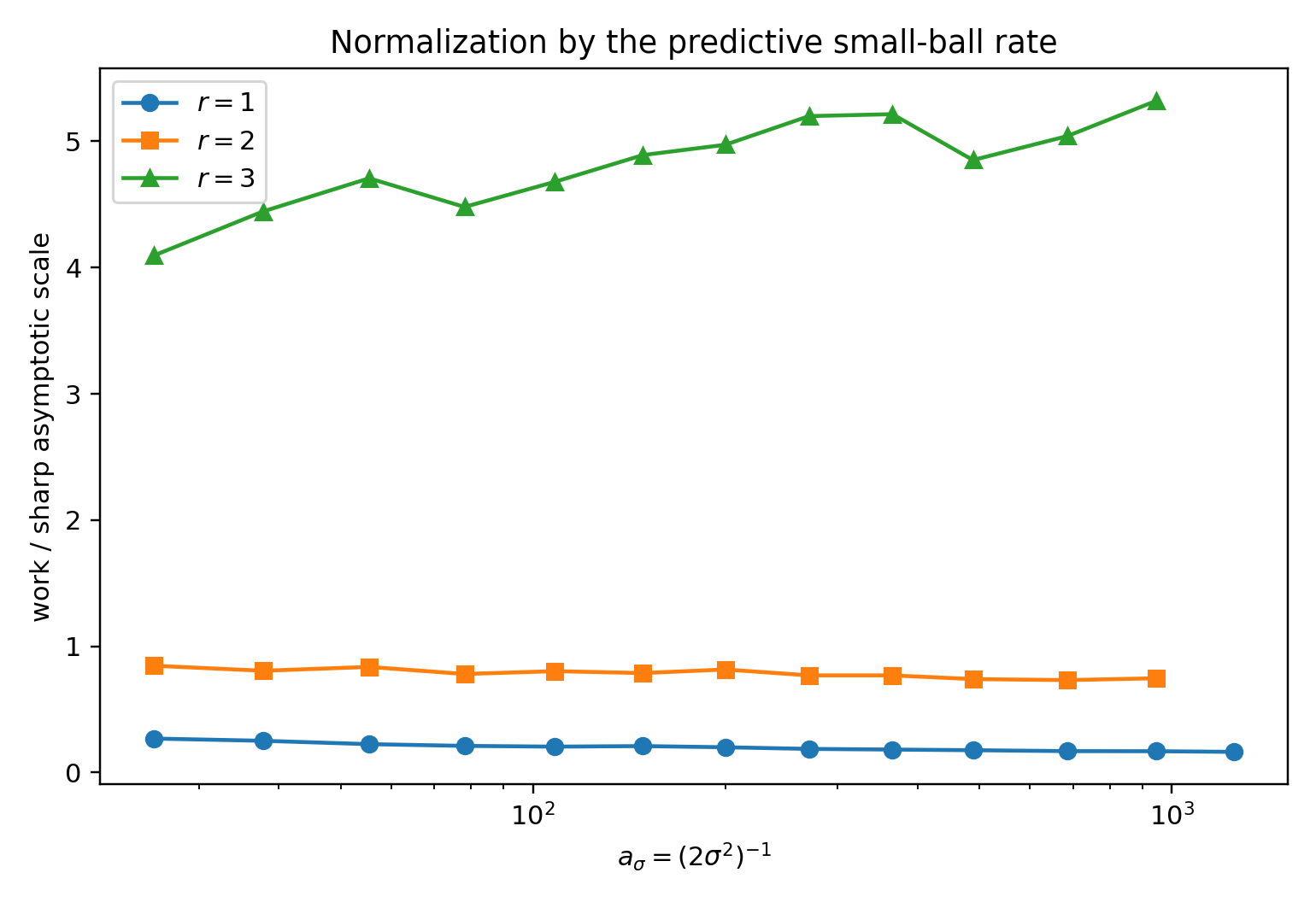}
\caption{Work normalized by $\gamma$ times the predicted sharp scale.}
\end{subfigure}
\caption{Sharp actual-work scaling under independent observations.}
\label{fig:sharp-work}
\end{figure}

The empirical scaling distinguishes the three regimes predicted by the sharp work theory. For $d_{\mathrm{pred}}=3$, the tail exponent $1.006$ is essentially the predicted linear rate. At the critical dimension $d_{\mathrm{pred}}=2$, removing the logarithmic factor gives a tail exponent $0.965$, close to the predicted value one. Convergence is slower for $d_{\mathrm{pred}}=1$: the fitted tail exponent $1.418$ remains below the asymptotic value $3/2$, and the normalized work retains a mild downward drift over the simulated range. The experiment measures the work required to process the master PPP rather than posterior Monte Carlo accuracy; consequently, the small number of retained atoms at the smallest noise levels does not enter the work estimate. This experiment deliberately isolates the independent scalar-factor architecture for which the sharp actual-work theorem is proved.

\FloatBarrier
\subsection{A function-valued heterogeneous reaction coefficient}

We next place the function-valued unknown directly in the elliptic operator. Let
\[
q_\xi(x)=40\left\{\frac12+\sum_{\ell=1}^{\infty}\omega_\ell\xi_\ell\cos(2\pi k_\ell\cdot x)\right\},
 \qquad \xi_\ell\stackrel{\rm iid}{\sim}\Unif[-1,1],
\]
where $\omega_\ell=A(1-\rho)\rho^{\ell-1}$, $A=0.48$, $\rho=0.5$. Then $0.8\le q_\xi(x)\le39.2$ almost surely and
\[
\left(32+q_\xi(x)-\frac12\Delta\right)u_{q_\xi}(x)=32,
 \qquad x\in\Torus.
\]

Eight observations are used with $\sigma=0.1$, and a fixed low-mode
field supplies synthetic data. Table~\ref{tab:hetero-forward-data} reports
the sensor data, continuum forward-coin estimates, and mean inspected killing
events per forward coin. The experiment tests agreement with a deterministic
reference posterior, rather than recovery of the full field from eight
observations.

\begin{table}[!tbp]
\centering\scriptsize\setlength{\tabcolsep}{3.2pt}
\caption{Function-valued reaction-coefficient example: observation data and continuum forward-oracle validation.}
\label{tab:hetero-forward-data}
\begin{tabular}{cccccc}
\toprule
sensor $x_j$ & FD64 & perturbation & $y_j$ & coin estimate & potential-killing events\\
\midrule
$(0.5,0.5)$ & 0.65030 & 0.008 & 0.65830 & 0.65053 & 0.811\\
$(0,0.5)$ & 0.52586 & -0.006 & 0.51986 & 0.52457 & 0.658\\
$(0,0)$ & 0.57319 & 0.006 & 0.57919 & 0.57190 & 0.722\\
$(0.5,0)$ & 0.75309 & -0.007 & 0.74609 & 0.75033 & 0.956\\
$(0.5,0.72)$ & 0.68927 & 0.006 & 0.69527 & 0.68933 & 0.863\\
$(0.5,0.28)$ & 0.68927 & -0.005 & 0.68427 & 0.68857 & 0.863\\
$(0,0.28)$ & 0.54389 & 0.004 & 0.54789 & 0.54167 & 0.675\\
$(0,0.72)$ & 0.54389 & -0.006 & 0.53789 & 0.54407 & 0.675\\
\bottomrule
\end{tabular}
\end{table}

Figure~\ref{fig:hetero-forward} compares the same continuum coin estimates with deterministic FD refinement.

\begin{figure}[!tbp]
\centering
\includegraphics[width=0.76\textwidth]{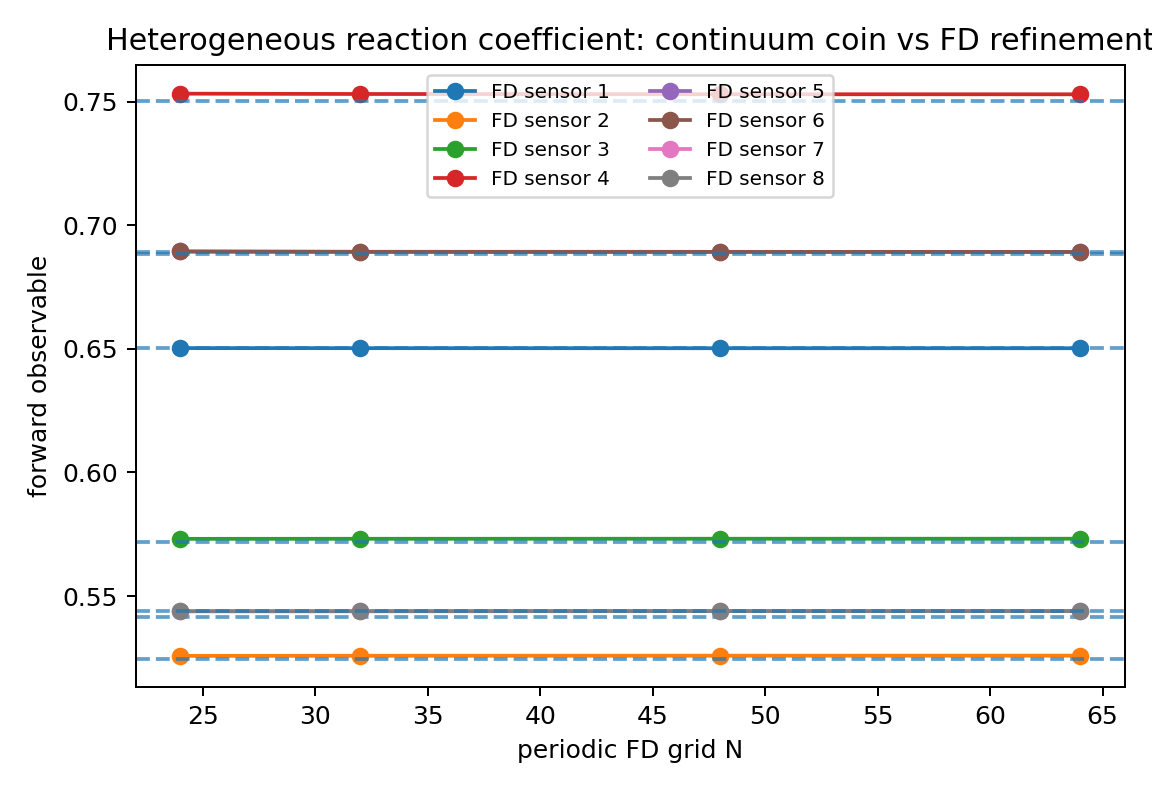}
\caption{Continuum forward-coin estimates (dashed levels) compared with deterministic FD refinement at the eight sensors.}
\label{fig:hetero-forward}
\end{figure}

The deterministic reference uses the first 24 modes; the omitted pointwise field amplitude is at most $2.86\times10^{-8}$. This truncation is used only for validation. The likelihood-coin sampler has no fixed coefficient cap. Each prior candidate starts with a cache of 16 iid coefficients. If the deterministic geometric tail enclosure does not resolve a field-value Bernoulli decision, fresh iid coordinates are appended one at a time from the product prior while the same auxiliary uniform decision is continued. Revelation stops as soon as the shrinking tail enclosure certifies that decision. Thus the implementation follows the infinite-series construction of Theorem~\ref{thm:lazy-function}: the stored prefix is a dynamically extensible cache, not a truncation of the posterior parameter. In the reported posterior run 290 new coefficients were generated on demand across 9446 master candidates, and the largest prefix required by any candidate was 28 coefficients. The average number of inspected coefficients per potential-killing event was about $1.92$.

The Bernstein order is selected from $n\in\{48,64,80,96,112,128\}$. Figure~\ref{fig:hetero-factory} shows the tradeoff between master-process inflation and coefficient-inspection work; the empirical score is minimized at $n=80$. The compensated likelihood estimates fluctuate around the FD64 reference value $0.88665$ without systematic drift in $n$.

\begin{figure}[!tbp]
\centering
\begin{subfigure}{0.49\textwidth}
\centering
\includegraphics[width=\textwidth]{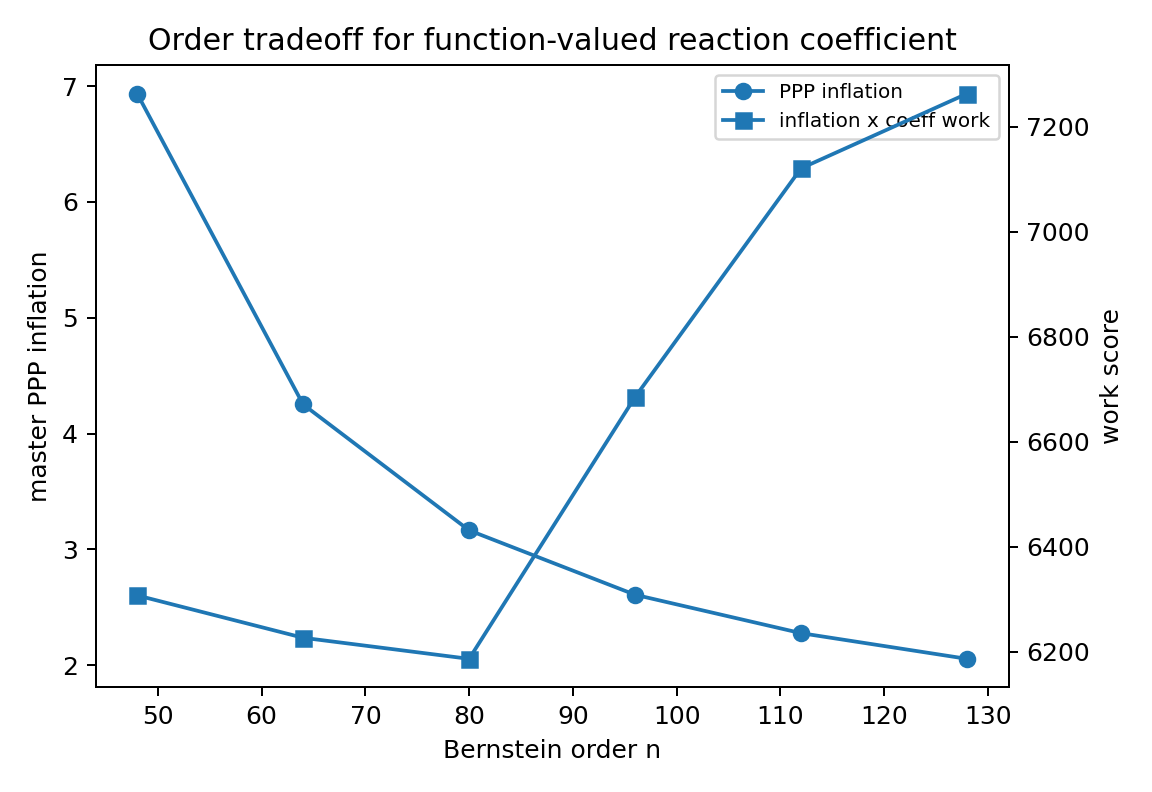}
\caption{Master inflation and empirical work score.}
\end{subfigure}\hfill
\begin{subfigure}{0.49\textwidth}
\centering
\includegraphics[width=\textwidth]{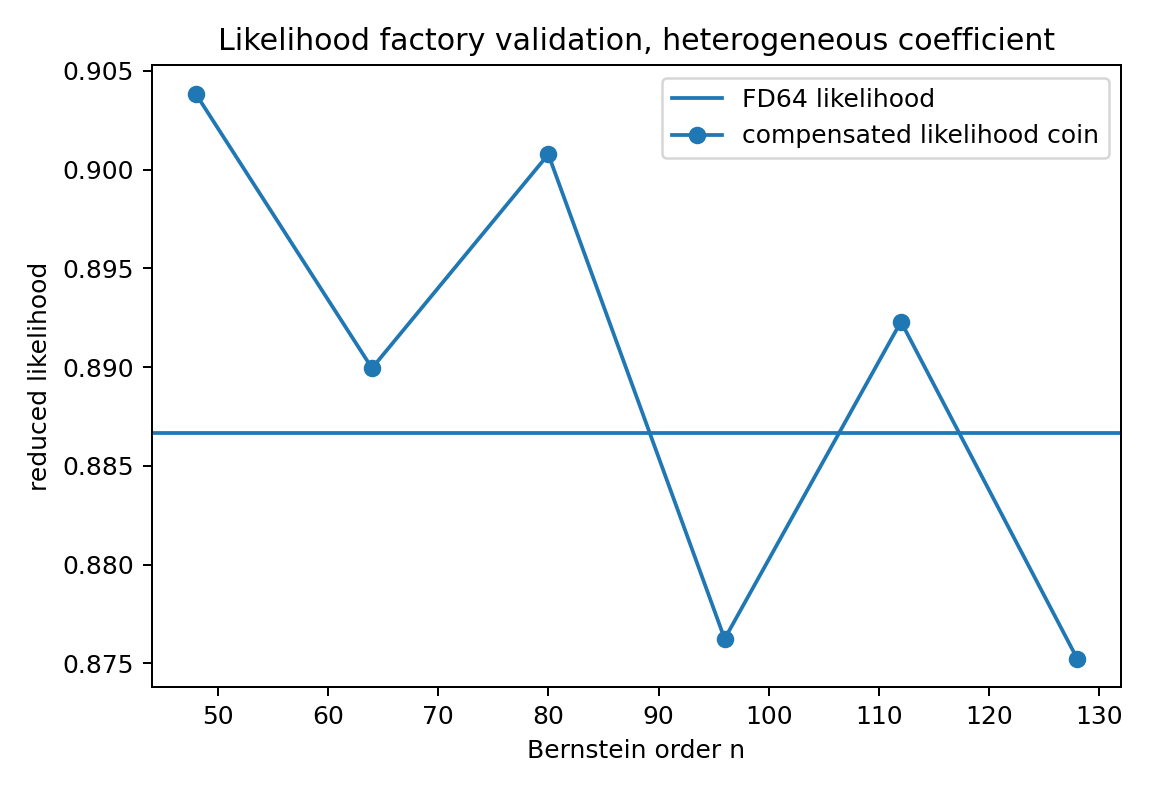}
\caption{Compensated likelihood estimates across orders.}
\end{subfigure}
\caption{Bernstein-order diagnostics for the function-valued example.}
\label{fig:hetero-factory}
\end{figure}

A $20\times20$ finite-difference importance calculation based on 5000 prior functions is compared with the posterior PPP in Table~\ref{tab:hetero-posterior}. The PPP uses $\gamma=3000$ and $n=80$; 9446 master candidates produce 658 retained atoms.

\begin{table}[!tbp]
\centering\small
\caption{Function-valued reaction-coefficient example: posterior validation against the finite-difference importance reference.}
\label{tab:hetero-posterior}
\begin{tabular}{lccccc}
\toprule
method & $Z$ & $E[\xi_1\mid y]$ & $E[\xi_2\mid y]$ & $E[\xi_3\mid y]$ & $E[\xi_4\mid y]$\\
\midrule
FD importance reference & 0.216810 & 0.62107 & -0.17511 & 0.00531 & -0.02136\\
posterior PPP & 0.219333 & 0.60203 & -0.13922 & -0.01602 & -0.01271\\
\bottomrule
\end{tabular}
\end{table}

Figure~\ref{fig:hetero-post} makes the validation criterion explicit. The coefficient plot includes truth markers only for context. The field plot compares the four-mode FD-reference posterior mean directly with the four-mode PPP posterior mean and their difference. The relative $L^2$ difference between these two projected posterior means is approximately $8.87\times10^{-3}$, with maximum pointwise difference about $0.396$. Their discrepancy from the synthetic truth instead reflects limited identifiability under the eight-observation design.

\begin{figure}[!tbp]
\centering
\begin{subfigure}{0.40\textwidth}
\includegraphics[width=\textwidth]{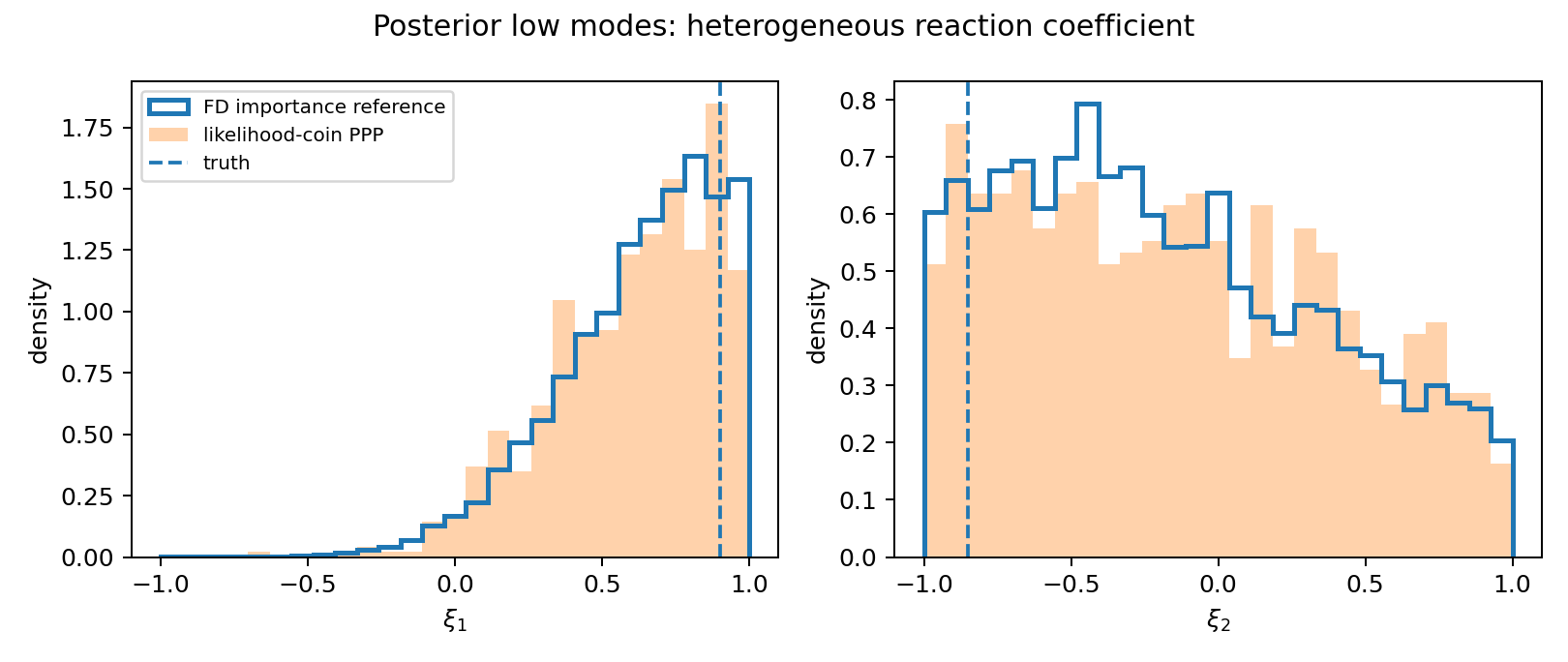}
\caption{First two posterior coefficient marginals.}
\end{subfigure}\hfill
\begin{subfigure}{0.58\textwidth}
\includegraphics[width=\textwidth]{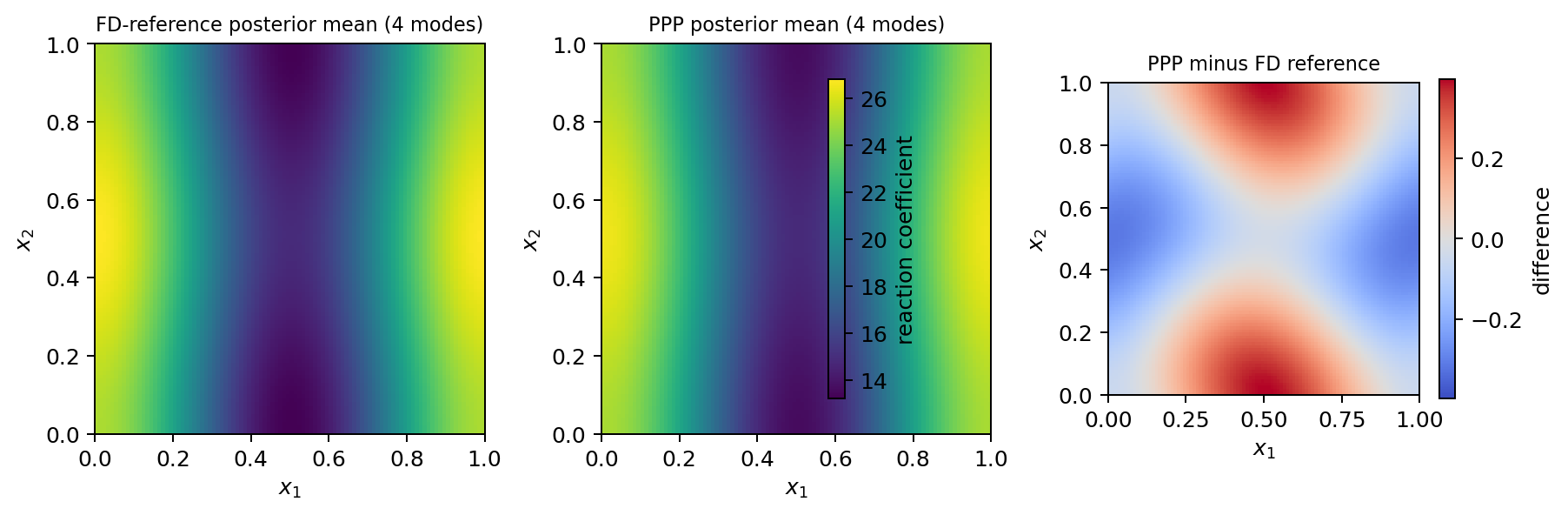}
\caption{Four-mode posterior-mean fields and their difference.}
\end{subfigure}
\caption{Posterior validation for the heterogeneous reaction coefficient.}
\label{fig:hetero-post}
\end{figure}

Overall, this experiment demonstrates that the continuum likelihood-coin
construction can be implemented for a genuinely function-valued unknown.
The stochastic sampler agrees with the deterministic reference at the
posterior level, while neither a spatial PDE discretization nor a fixed
truncation of the random field is introduced into the posterior target.
The next experiment examines the computational consequence of the former
property by comparing continuum sampling with finite-difference computation
at matched posterior accuracy.


\subsection{Matched-accuracy cost against deterministic PDE rejection}
\label{subsec:cost-benchmark}

To assess the computational effect of removing deterministic forward discretization, we compare the continuum likelihood-coin sampler with deterministic FD prior rejection at the \emph{same target RMS error} for the posterior mean of the first four coefficients in the heterogeneous reaction-coefficient problem. This comparison uses the same observation locations, the same data, and bilinear periodic interpolation at every FD resolution. A $64\times64$ FD posterior is used as a high-resolution numerical reference for the benchmark; it is not treated as an exact continuum posterior. The benchmark uses a common pool of 5000 prior draws across FD grids and ten continuum timing replicates. Its purpose is to compare accuracy--cost mechanisms on a common implementation rather than to estimate a hardware-independent threshold.

Let $\mu_\star$ and $\Sigma_\star$ denote the reference posterior mean and covariance of $(\xi_1,\ldots,\xi_4)$. For $m$ independent continuum posterior atoms, the benchmark uses
\begin{equation}
 \operatorname{RMSE}_{\rm coin}^2(m)\approx \frac{\operatorname{tr}(\Sigma_\star)}{m}.
 \label{eq:rmse-coin-benchmark}
\end{equation}
For an FD grid $N$, let $\mu_N$, $\Sigma_N$, and $Z_N$ be the corresponding discretized posterior mean, covariance, and reduced evidence. Its projected rejection error and cost are
\begin{equation}
 \operatorname{RMSE}_{N}^2(m)\approx
 \|\mu_N-\mu_\star\|_2^2+\frac{\operatorname{tr}(\Sigma_N)}{m},
 \qquad
 C_N(m)=m\,\frac{t_N}{Z_N},
 \label{eq:rmse-fd-benchmark}
\end{equation}
where $t_N$ is the measured mean time for one deterministic PDE likelihood evaluation. Thus a grid is infeasible for a requested tolerance $\varepsilon$ whenever its estimated discretization bias $\|\mu_N-\mu_\star\|_2$ already exceeds $\varepsilon$. The factor $Z_N^{-1}$ makes the prior-rejection penalty explicit on the deterministic side as well.

The continuum timing is measured from ten independent likelihood-coin runs after JIT compilation, each with PPP scale $\gamma=3000$, and pooled per retained atom. The deterministic benchmark uses a common pool of 5000 prior draws at grids $N\in\{16,20,24,32,40,48,64\}$, which reduces noise in differences between grid-dependent posterior means. These are implementation-level timings on one machine, not hardware-independent complexity constants.

The enlarged common pool gives a monotone decrease in the estimated four-mode posterior-mean bias, from $7.39\times10^{-4}$ at $N=16$ to $4.4\times10^{-5}$ at $N=48$. Across the ten continuum timing replicates, the pooled cost is $3.376\times10^{-3}$ seconds per retained atom; the replicate wall times have a coefficient of variation of about $1.2\%$. These diagnostics indicate that the qualitative accuracy--cost comparison is not driven by a single timing realization.

\begin{figure}[!tbp]
\centering
\includegraphics[width=0.72\textwidth]{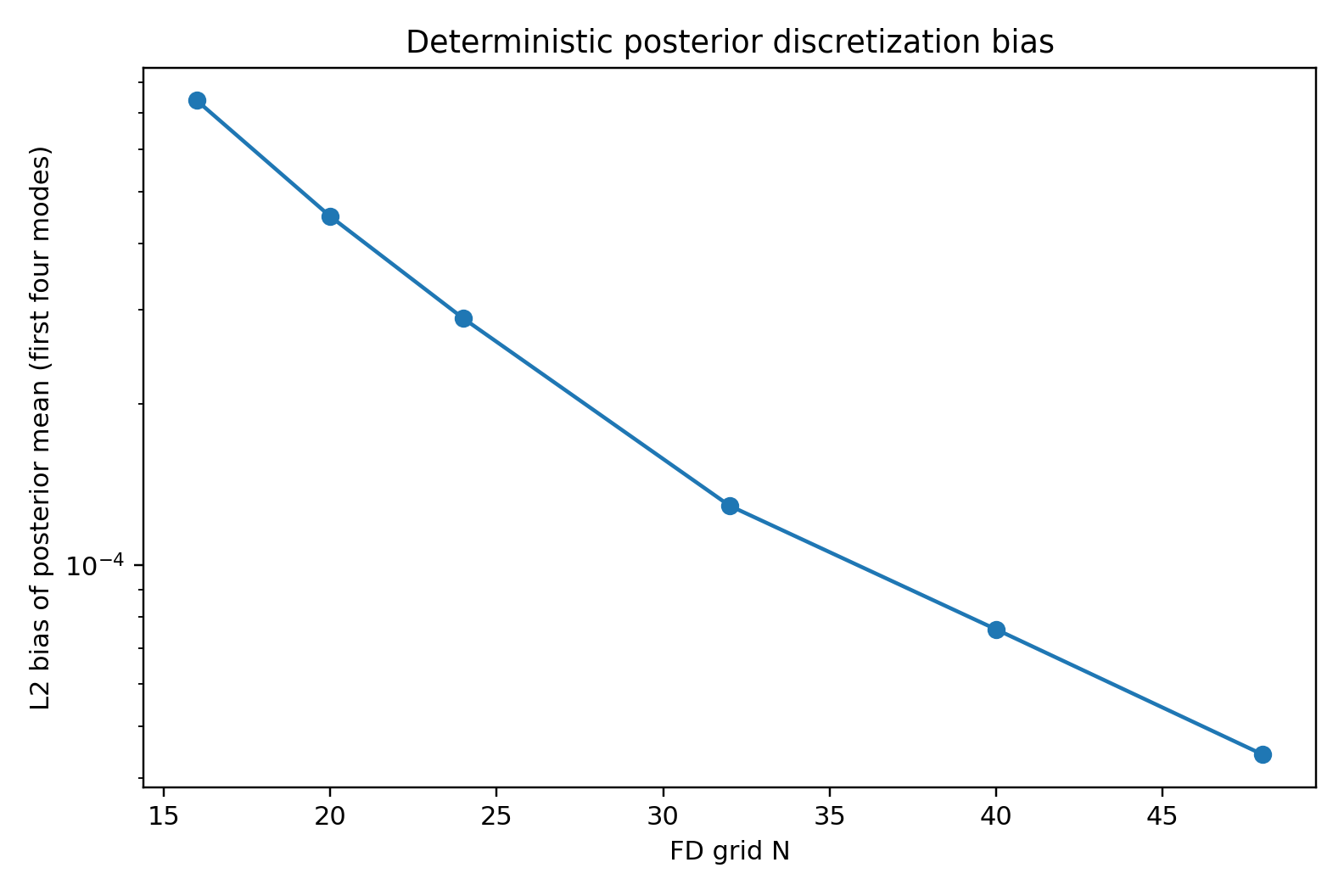}
\caption{Estimated FD posterior-mean discretization bias relative to the $64\times64$ reference.}
\label{fig:posterior-discretization-bias}
\end{figure}

Figure~\ref{fig:posterior-discretization-bias} shows the deterministic bias component that was absent from a comparison based only on time per retained draw. At coarse resolution this bias is small enough for moderate accuracy targets, so inexpensive deterministic solves dominate. At tighter tolerances, however, a fixed grid ceases to be admissible regardless of how many Monte Carlo samples are drawn.

\begin{figure}[!tbp]
\centering
\includegraphics[width=0.80\textwidth]{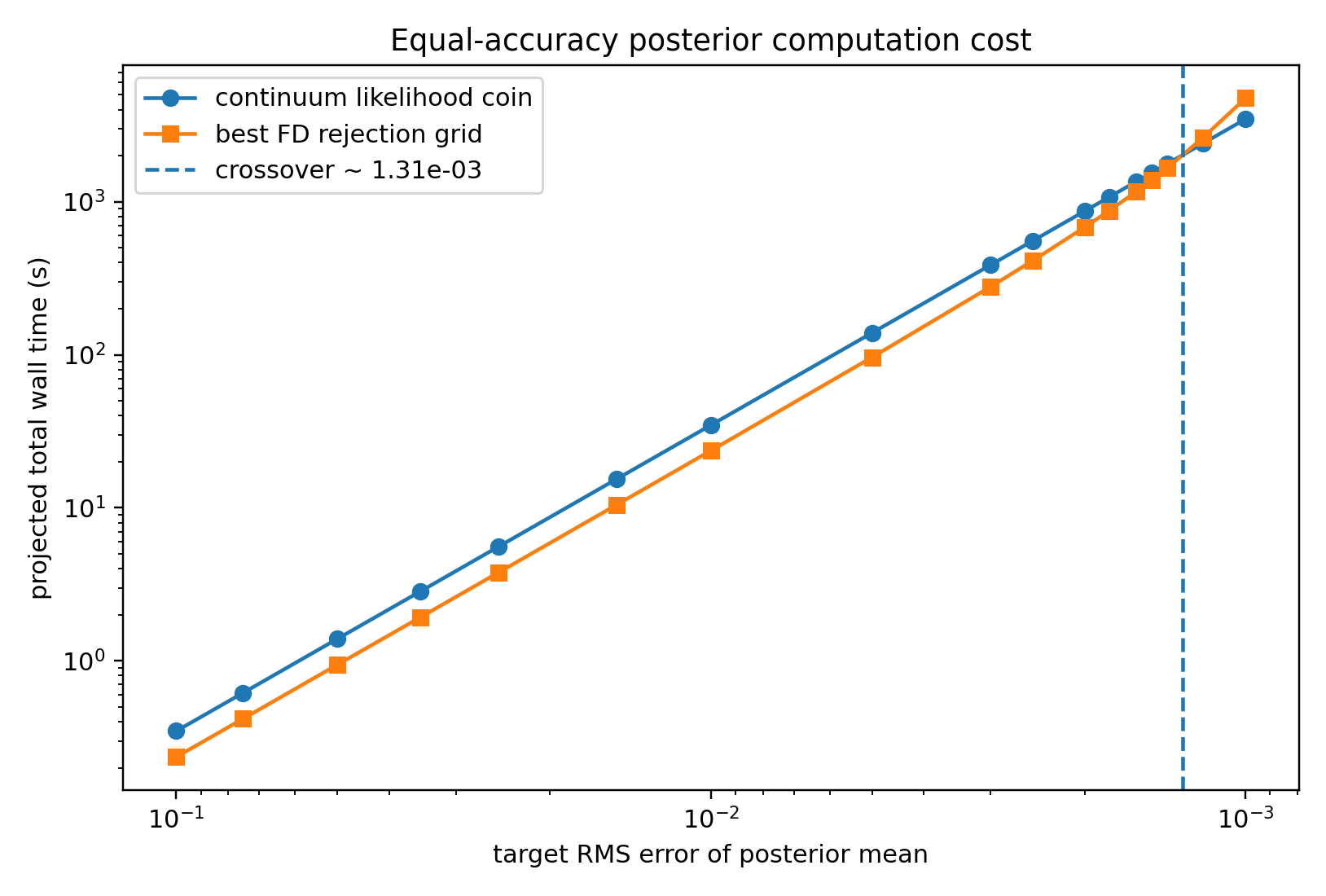}
\caption{Projected wall time at matched posterior-mean RMS error. The dashed line marks the implementation-specific crossover near $1.31\times10^{-3}$.}
\label{fig:equal-accuracy-cost}
\end{figure}

The matched-accuracy comparison in Figure~\ref{fig:equal-accuracy-cost} shows a regime-dependent tradeoff. For moderate tolerances, the inexpensive $N=16$ FD rejection calculation is cheaper. As the target approaches the $N=16$ discretization-bias floor, however, the number of posterior draws required by \eqref{eq:rmse-fd-benchmark} grows rapidly even before that grid becomes infeasible. The projected cost curves cross at approximately $1.31\times10^{-3}$. At the tightest target shown, $10^{-3}$, the cost-minimizing admissible grid has refined to $N=20$. The continuum method has no spatial-discretization bias term in \eqref{eq:rmse-coin-benchmark}. We regard the crossing as an implementation-specific indication of the accuracy regime in which this distinction matters, not as a general speed advantage: its location depends on the PDE solver, implementation, hardware, prior, data, and accuracy functional.

The comparison is intentionally focused: it contrasts two independent-sampling constructions while accounting for deterministic discretization bias in the accuracy budget. It is not a benchmark against optimized multilevel, surrogate, transport, SMC, or MCMC methods. The $64\times64$ reference and the estimated bias/covariance quantities also carry finite-sample and residual discretization error. Within this controlled comparison, the experiment shows how the absence of a spatial-discretization bias term changes the accuracy--cost balance as the requested tolerance is tightened.

\section{Conclusion}
\label{sec:conclusion}

We introduced an exact direct-sampling framework for Bayesian inverse problems with a bounded Bernoulli-access interface. Forward Bernoulli events are converted into scaled Gaussian likelihood coins, and Poisson thinning produces posterior atoms that are iid conditional on the retained count. For independent Gaussian observations, the implemented early-stopped factory admits an exact mean-work identity and sharp small-noise regimes determined by local prior-predictive geometry. A factorial-moment construction shows how the likelihood factory can also accommodate correlated Gaussian errors.

The construction separates the statistical-computational mechanism from the model-specific realization of the forward oracle. For the bounded elliptic resolvent class considered here, Feynman--Kac sampling, Poisson killing, and lazy random-series decisions provide continuum Bernoulli access for a function-valued coefficient. In this realization, neither a deterministic spatial mesh nor a fixed truncation of the random field is part of the posterior target. The numerical experiments validate the forward and likelihood coins, support the predicted stopped-work regimes, and show agreement between the likelihood-coin posterior sampler and independent deterministic reference calculations.

The matched-accuracy benchmark further separates Monte Carlo error from deterministic posterior discretization bias. In the tested implementation, coarse finite-difference rejection is cheaper for moderate accuracy targets, whereas its projected cost increases as the requested tolerance approaches the coarse-grid bias floor and eventually requires grid refinement. The continuum sampler has no corresponding spatial-discretization term in its posterior error budget. The observed crossover is therefore best interpreted as a problem- and implementation-specific illustration of this tradeoff, rather than as a universal speed comparison.

The main outcome is a reusable route from Bernoulli-accessible forward observables to exact posterior samples, together with a work analysis that reflects the early stopping used in the implementation. Extending the class of forward representations and likelihood models admitting practical Bernoulli factories, and sharpening adaptive order-selection strategies for the resulting exact samplers, are natural directions for further work.

\appendix

\section{Bernstein compensation details}
\label{app:bernstein-compensation}

For the normalized scalar factors used in Section~\ref{sec:factory}, the auxiliary normalization in \eqref{eq:ck} is in fact inactive. Since $\beta_{k,n}(y)$ is convex in $k$, the maximum of $c_{k,n}(y)=\beta_{k,n}(y)-\min_j\beta_{j,n}(y)$ is attained at $k=0$ or $k=n$. For every $k\in\{0,\ldots,n\}$,
\[
 \beta_{0,n}(y)-\beta_{k,n}(y)
 =\frac{k\{2y(n-1)-(k-1)\}}{n(n-1)}
 \le \frac{k(2n-k-1)}{n(n-1)}\le1.
\]
Hence $c_{0,n}(y)\le1$. The symmetry
$\beta_{n-k,n}(1-y)=\beta_{k,n}(y)$ gives $c_{n,n}(y)\le1$ as well. Therefore
\[
 0\le c_{k,n}(y)\le1,
 \qquad C_n(y)=1,
 \qquad y\in[0,1].
\]
The factor $C_n$ is retained in the main formulas because it makes the Bernoulli admissibility explicit and allows harmless rescalings of the scalar construction without changing the derivation.

The centered scalar factor is a useful special case of the same construction. If a particular component has $y=0$, then $b_n(0)=0$, and already at $n=2$,
\[
 \beta_{K,2}(0)=\frac{K(K-1)}2.
\]
This variable equals one if and only if both input $p$-coins are one. The Poisson zero-event transform therefore reduces to a Poisson number of independent $p^2$-events and has success probability $e^{-ap^2}$. This simplification concerns one scalar likelihood factor; the observation model in the paper remains the vector model \eqref{eq:obsmodel}.

\section{Continuous safe shift for correlated Gaussian errors}
\label{app:correlated-shift}
Let $\delta_j=\Omega_{jj}/(n_j-1)$ and
\[
\widetilde\beta_{\bm n}^{\Omega}(z;y)
=(z-y)^\top\Omega(z-y)-\sum_{j=1}^J\delta_j z_j(1-z_j),
\qquad z\in[0,1]^J.
\]
The count grid is contained in $[0,1]^J$, so the continuous minimum is no larger than the discrete one and hence $\bar b_{\bm n}^{\Omega}(y)\ge b_{\bm n}^{\Omega}(y)$. Moreover,
\[
\nabla^2\widetilde\beta_{\bm n}^{\Omega}
=2\Omega+2\operatorname{diag}(\delta_1,\ldots,\delta_J)\succ0,
\]
so the box-constrained problem is strictly convex. Put
$S_{\bm n}(z)=\sum_j\delta_j z_j(1-z_j)$. Since $t(1-t)$ is one-Lipschitz on $[0,1]$,
$|S_{\bm n}(z)-S_{\bm n}(y)|\le O(n_{\min}^{-1})\|z-y\|_2$, while
$(z-y)^\top\Omega(z-y)\ge\lambda_{\min}(\Omega)\|z-y\|_2^2$.
Optimizing this quadratic bound gives
\[
\bar b_{\bm n}^{\Omega}(y)
=\sum_{j=1}^J\frac{\Omega_{jj}y_j(1-y_j)}{n_j}+O(n_{\min}^{-2}).
\]
Evaluating at a nearest count-grid point gives the same expansion for $b_{\bm n}^{\Omega}(y)$.

\section{Small-ball integral asymptotics}
\label{app:small-ball}

The following standard regular-variation lemma supplies the analytic step used
in Theorem~\ref{thm:sharp-actual-work} and
Corollary~\ref{cor:work-per-atom}; see, for example,
\cite{BinghamGoldieTeugels1987}.

\begin{lemma}

Let $Q:\X\to[0,\infty)$ be measurable and suppose
$F(t):=\mu_0\{Q\le t\}=ct^\alpha\{1+o(1)\}$ as $t\downarrow0$, with
$c>0$ and $\alpha>0$. Let $w:\X\to[0,\infty)$ be bounded and assume that
for some $w_0>0$,
\[
 \sup_{Q(q)\le t}|w(q)-w_0|\longrightarrow0
 \qquad(t\downarrow0).
\]
If $a\to\infty$, $b_a\ge0$, and $ab_a\to\beta\in(0,\infty)$, then
\[
\int w(q)\frac{1-e^{-a(b_a+Q(q))}}{b_a+Q(q)}\,\mu_0(\mathrm dq)
\sim
\begin{cases}
 c w_0\Psi_\alpha(\beta)a^{1-\alpha},&0<\alpha<1,\\[0.8ex]
 c w_0\log a,&\alpha=1,
\end{cases}
\]
where $\Psi_\alpha$ is defined in \eqref{eq:Psi-alpha}. If $\alpha>1$,
then $Q^{-1}$ is locally integrable and
\[
\int w(q)\frac{1-e^{-a(b_a+Q(q))}}{b_a+Q(q)}\,\mu_0(\mathrm dq)
\longrightarrow
\int \frac{w(q)}{Q(q)}\,\mu_0(\mathrm dq)<\infty.
\]
Moreover,
\[
 \int e^{-aQ(q)}\,\mu_0(\mathrm dq)
 \sim c\Gamma(\alpha+1)a^{-\alpha}.
\]
\end{lemma}

\begin{proof}
Push $\mu_0$ forward by $Q$. The assumption on $w$ implies that the weighted
sublevel measure
\[
 F_w(t):=\int_{\{Q\le t\}}w(q)\,\mu_0(\mathrm dq)
\]
satisfies $F_w(t)\sim cw_0t^\alpha$ under the stated assumption $w_0>0$. For $0<\alpha<1$,
Karamata's theorem for Stieltjes transforms, applied after the rescaling
$t=aQ$, gives
\[
 a^{\alpha-1}
 \int\frac{1-e^{-a(b_a+Q)}}{b_a+Q}\,\mathrm dF_w(Q)
 \longrightarrow
 c w_0\alpha\int_0^\infty
 t^{\alpha-1}\frac{1-e^{-(\beta+t)}}{\beta+t}\,\mathrm dt,
\]
which is the first displayed asymptotic. At $\alpha=1$, split the integral
at a fixed small $\delta$. On $a^{-1}\ll Q\le\delta$ the kernel is
$Q^{-1}\{1+o(1)\}$, and Stieltjes integration by parts with
$F_w(t)\sim cw_0t$ gives $cw_0\log a\{1+o(1)\}$; the regions
$Q=O(a^{-1})$ and $Q\ge\delta$ contribute only $O(1)$.

If $\alpha>1$, Stieltjes integration by parts gives
\[
 \int_{\{0<Q\le\delta\}}Q^{-1}\,\mu_0(\mathrm dq)
 =\frac{F(\delta)}{\delta}
 +\int_0^\delta\frac{F(t)}{t^2}\,\mathrm dt<\infty.
\]
The kernel is bounded by $Q^{-1}$ and converges pointwise to $Q^{-1}$, so
dominated convergence yields the high-dimensional limit. Finally,
the Laplace asymptotic is the standard Karamata Laplace-transform
asymptotic for a distribution function regularly varying at zero with index
$\alpha$.
\end{proof}

\section*{Acknowledgments}
Correspondence may be addressed to Zhiliang Deng.

\bibliographystyle{plainnat}
\bibliography{references}

\end{document}